\documentclass[pdf, 10pt]{amsart}

\usepackage{amsthm, amsmath, amssymb,}% natbib}
\usepackage[OT1]{fontenc}
\usepackage[colorlinks,citecolor=blue,urlcolor=blue]{hyperref}
\usepackage{color}

\newtheorem{thm}[equation]{Theorem}
\newtheorem{con}[equation]{Condition}
\newtheorem{prop}[equation]{Proposition}

\newtheorem{ass}[equation]{Assumption}
\newtheorem{lem}[equation]{Lemma}

\newtheorem{rem}[equation]{Remark}

\def\R{{\mathbb{R}}}
\def\C{{\mathbb{C}}}
\def\N{{\mathbb{N}}}

\def\E{{\mathbb{E}}}
\def\P{{\mathbb{P}}}
\def\Ss{{\mathbb{S}}}
\def\B{{\mathcal{B}}}
\def\Cc{{\mathcal{C}}}
\def\D{{\mathcal{D}}}
\def\F{{\mathcal{F}}}
\def\H{{\mathcal{H}}}
\def\L{{\mathcal{L}}}
\def\Rr{{\mathcal{R}}}
\def\S{{\mathcal{S}}}

\def\bf{\bar f}
\def\bm{\bar \mu}
\def\bp{\overline{\psi}}
\def\ct{\mathcal{C}_{\rho}}
\def\bl1{\mathcal{B}_{\rho,\epsilon,\lambda}}
\def\el{{\epsilon, \lambda}}
\def\rel{{\rho,\epsilon,\lambda}}

\newcommand{\supp}{\mbox{supp}}
\newcommand{\Dis}{\mbox{Dis}}

\newcommand{\al}{\alpha}
\newcommand{\be}{\beta}
\newcommand{\ga}{\gamma}
\newcommand{\de}{\delta}
\newcommand{\ka}{\kappa}
\newcommand{\la}{\lambda}
\newcommand{\La}{\Lambda}
\newcommand{\ep}{\epsilon}
\newcommand{\eps}{\varepsilon}
\newcommand{\p}{\psi}
\newcommand{\si}{\sigma}
\newcommand{\vp}{\varphi}
\newcommand{\te}{\theta}
\newcommand{\Te}{\Theta}
\newcommand{\8}{\infty}
\newcommand{\ups}{\Upsilon}
\newcommand{\Dt}{\Delta_t}

\def\bf{\bar f}
\def\bm{\bar \mu}
\def\bp{\overline{\psi}}

\numberwithin{equation}{section}

\begin{document}

\title[Heavy tail phenomenon and
convergence to stable laws]{Heavy tail phenomenon and convergence to stable laws for iterated Lipschitz maps}
\dedicatory{\normalsize In memory of Andrzej Hulanicki}
\author[M. Mirek]
{Mariusz Mirek}
\address{M.Mirek\\
University of Wroclaw\\ Institute of Mathematics\\
 50-384 Wroclaw\\
 Poland}
\email{mirek@math.uni.wroc.pl}

\thanks{
This research project has been partially supported by Marie Curie
Transfer of Knowledge Fellowship Harmonic Analysis, Nonlinear Analysis and Probability (contract
number MTKD-CT-2004-013389) and by KBN grant N201 012 31/1020.}

\begin{abstract}
We consider the Markov chain $\{X_n^x\}_{n=0}^\infty$ on $\R^d$ defined by the stochastic recursion
$X_{n}^{x}=\p_{\theta_{n}}(X_{n-1}^{x})$, starting at $x\in\R^d$, where $\theta_{1}, \theta_{2},
\ldots$ are i.i.d. random variables taking their values in a metric space $(\Theta, \mathfrak{r}),$
and $\p_{\theta_{n}}:\R^d\mapsto\R^d$ are Lipschitz maps. Assume that the Markov chain has a unique
stationary measure $\nu$. Under appropriate assumptions on $\p_{\theta_n}$, we will show that the
measure $\nu$ has a heavy tail with the exponent $\alpha>0$ i.e. $\nu(\{x\in\R^d: |x|>t\})\asymp
t^{-\alpha}$. Using this result we show that properly normalized Birkhoff sums $S_n^x=\sum_{k=1}^n
X_k^x$, converge in law to an $\alpha$--stable law for $\alpha\in(0, 2]$.
\end{abstract}
\maketitle
\section{Introduction and Statement of Results}
We consider the Euclidean space $\R^{d}$ endowed with the scalar product $\langle
x,y\rangle=\sum_{i=1}^{d}x_{i}y_{i}$, the norm $|x|=\sqrt{\langle x, x\rangle},$ and its Borel
$\si$--field $\B or(\R^d)$. An iterated random function is a sequence of the form
\begin{align}
\label{rec}X_{n}^{x}=\p(X_{n-1}^{x}, \te_{n}),
\end{align}
where $n\in\N$, $X_{0}^{x}=x$ and $\te_{1}, \te_{2},\ldots \in \Te$ are independent and identically
distributed according to the measure $\mu$ on a metric space $\Te=(\Te, \mathfrak{r})$. We assume
that $\p:\R^d\times\Te\mapsto\R^d$ is jointly measurable and we write $\p_{\te}(x)=\p(x, \te)$.
Then the sequence $(X_{n}^{x})_{n\geq0}$ is a Markov chain with the state space $\R^d$, the initial
Dirac distribution $\delta_x,$ and the transition probability $P$ defined by $P(x,
B)=\int_{\Te}\mathbf{1}_{B}(\p_{\te}(x))\mu(d\te)$ for all $x\in\R^d$ and $B\in \B or(\R^d)$.
Unless otherwise stated we assume throughout this paper that for every $\te\in\Te$,
$\p_{\te}:\R^{d}\mapsto \R^{d}$ is a Lipschitz map with the Lipschitz constant
\begin{align*}
L_{{\te}}=\sup_{x\not= y}\frac{|\p_{\te}(x)- \p_{\te}(y)|}{|x-y|}<\8.
\end{align*}

Matrix recursions
\begin{equation}\label{matrix}
X_{n}^x=\p_{\te_n}(X_{n-1}^x)=M_nX_{n-1}^x+Q_n\in \R ^d,
\end{equation}
where $\te_n=(M_n, Q_n)\in Gl(\R^d)\times\R^d=\Te$ and
$X_{0}^{x}=x\in\R^{d}$ ($Gl(\R^d)$ is the group of $d\times d$
invertible matrices) are probably the best known examples of the
situation we have in mind \cite{BDGHU, DF, Gui, K, K1}.

If the Lipschitz constant $L_{{\te}}$ is contracting in average i.e. $\int_{\Te}\log
(L_{{\te}})\mu(d\te)<0$ and $\int_{\Te}|\log(L_{{\te}})|+\log^+(|\p_{\te}(x_0)|)\mu(d\te)<\8$ for
some $x_0\in\R^d$, then $(\ref{rec})$ has a unique (in law) stationary solution $S$ with law $\nu$.
In fact, $S=\lim_{n\rightarrow\8}\p_{\te_1}\circ \p_{\te_2}\circ\ldots\circ \p_{\te_n}(x)$ a.s.
and does not depend on the starting point $x\in\R^d$ (see \cite{DF}, \cite{L} for more details).
Throughout this paper we assume that our Lipschitz maps $\p_{\te}$'s satisfy above conditions and
recursion $(\ref{rec})$ has the stationary solution $S$ with law $\nu$.

We are going to describe the asymptotic behavior of Birkhoff sums
$S_n^x=\sum_{k=1}^{n}X_{k}^{x}$ of (non independent) random
variables $X_{k}^{x}$. We prove that the sequence $S_n^x$
normalized appropriately converges to a stable law (see Theorem
\ref{limitthm}).

The problem has been recently studied in \cite{BDG} for the recursion \eqref{matrix} with $M\in \R
_+ ^{\ast}\times O(\R^d)$ ($O(\R^d)$ is the orthogonal group) and a central limit theorem has been
proved. Depending on the growth of $M$ and $Q$, a stable law or a Gaussian law appear as the limit.
In the first case the heavy tail behavior of the stationary solution of (\ref{matrix}) at infinity
is vital for the proof. (See \cite{BDGHU}).

On the other hand being linear is not that crucial for $\psi
 _{\te}$ and so, it is tempting to generalize the result of
 \cite{BDG} for a larger class of possible $\psi _{\te}$.
 Lipschitz transformations fit perfectly into the scheme --
see examples in section 2 due to Goldie \cite{G} and Borkovec and Kl\"uppelberg \cite{BK}.

To give an idea of our result let us formulate it in the special case of the recursion
\begin{align}\label{recmax}
X_{n}^{x}=\max(M_nX_{n-1}^{x}, Q_n),
\end{align}
where $(M_n, Q_n)_{n\in\N}\subseteq\R_+\times\R$ and $X_{0}^{x}=x\in\R$. For the stationary
solution $S$ of (\ref{recmax}) with law $\nu$, $\lim _{t\to \8} t^{\al}\P (\{S>t\})$ exists, and
under appropriate assumption, it is positive \cite{G}. Then the limit Theorem \ref{limitthm} is:

\begin{thm}
Assume that $(M_n, Q_n)_{n\in\N}\subseteq\R_+\times\R$ is the sequence of i.i.d. pairs with the law
$\mu$ such that $\E(M^{\al})=1,$ and $\E(M^{\al}|\log M|)<\8$, for some $\al\in(0, 2]$, the
conditional law of $\log|M|$, given $M\not=0$ is non arithmetic, $\P(Q>0)>0,$ and
$\E(|Q|^{\al})<\8$. Let $S_n^x=\sum_{k=1}^{n}X_{k}^{x}$ for $n\in\N$. Then given $0<\al <2$ (for
simplicity we assume here $\al \neq 1$), there is a sequence $d_n=d_n(\al)$ and a constant
$C_{\al}\in\C$ such that the random variables $n^{-\frac{1}{\al}}\left(S_n^x-d_n\right)$ converge
in law to the $\al$ -- stable random variable with the characteristic function
\begin{align*}
\ups_{\al}(t)=\exp(C_{\al}t^{\al}),\ \ \ \mbox{for $t>0$},
\end{align*}
If $\al=2$, there is a sequence $d_n=d_n(2)$ and a constant $C_{2}\in\R$ such that the random
variables $\left(n\log n\right)^{-\frac{1}{2}}\left(S_n^x-d_n\right)$ converge in law to the random
variable with characteristic function
\begin{align*}
\ups_{2}(t)=exp(C_{2}t^2),\ \ \ \mbox{for $t>0$}.
\end{align*}
If $\al\in(0, 1)$, then $d_n=0$, and if $\al\in(1, 2]$, then $d_n=nm$, where
$m=\int_{\R^d}x\nu(dx)$. Furthermore, $\Re C_{\al}<0$ for every $\al\in(0, 2]$.
\end{thm}

The paper is divided into three parts. In the first one (section 3) we describe the support of the
stationary law $\nu $ of (\ref{rec}) in the terms of the fixed points for maps $\p_{\te}$ -- see
Theorem \ref{support}. Secondly, in section 4 we take care of the tail of $\nu$. (See Theorem
\ref{HTthm} saying that $\nu(\{x\in\R^d: |x|>t\})\asymp t^{-\alpha}$). Finally, sections 5 and 6
are devoted to the proof of the limit Theorem \ref{limitthm}. The limit law is a stable law with
exponent $\al\in(0, 2]$. Thus we generalize the results for $"ax+b"$ model stated in \cite{GL1} and
\cite{BDG} for one dimensional and multidimensional situation respectively. The case where $\al>2$
has been widely investigated in the general context of complete separable metric spaces by
\cite{Be, HH1, HH2, HP, MW} and \cite{WW}. Recently, in \cite{JKO} the authors proved $\al$ --
stable theorem for $\al\in(0, 2)$ for additive functionals on metric spaces using martingale
approximation method, but our situation does not fit into their framework. Convergence to stable
laws were also studied by \cite{BG} and \cite{Gou}.

Now we are ready to formulate assumptions and to state theorems.

\subsection{Heavy tail phenomena}
In this section we state conditions that guarantee a heavy tail of $\nu$. Contrary to the affine
recursion
\begin{equation}\label{affine}
X_{n}^x=\p_{\te_n}(X_{n-1}^x)=M_nX_{n-1}^x+Q_n\in\R,
\end{equation}
where $\te_n=(M_n, Q_n)\in \R\times\R=\Te$, we need more than just the behavior of the Lipschitz
constant $L_{\te}$.

\begin{ass}[\textbf{Shape of the mappings $\p$}]\label{shapefun} For every $t>0$, let
$\p_{\te, t}:\R^d\mapsto\R^d$ be defined by $\p_{\te, t}(x)=t\p_{\te}(t^{-1}x)$, where $x\in\R^d$
and $\te\in\Te$. $\p_{\te, t}$ are called dilatations of $\p_{\te}$.
\begin{enumerate}
\item[(H1)] For every $\te\in\Te$, there exists a map $\bp_{\te}:\R^d\mapsto\R^d$ such that
$\lim_{t\rightarrow0}\p_{\te, t}(x)=\bp_{\te}(x)$ for every $x\in\R^d$, and $\bp_{\te}(x)=M_{\te}x$
for every $x\in \supp\nu$. The random variable $M_{{\te}}$ takes its values in the group
$G=\R_{+}^{\ast}\times K$, where $K$ is a closed subgroup of orthogonal group $O(\R^d)$.
\item[(H2)] For every $\te\in\Te$, there is a random variable $N_{\te}$
such that $\p_{\te}$ satisfies a cancellation condition i.e.
$|\p_{\te}(x)-M_{\te}x|\leq|N_{\te}|,$ for every $x\in\supp\nu$.
\end{enumerate}
\end{ass}

To get the idea what is the meaning of (H1)--(H2) the reader may think of the affine recursion
(\ref{affine}) with $\te=(M, Q) \in G\times\R^d=\Te$ or the recursion $\psi _{\te}(x)=\max\{Mx,
Q\}$, where $\te=(M, Q)\in\R_{+}^{\ast}\times\R=\Te$ (see section 2). Then $\bp_{\te}(x)=Mx$ or
$\bp_{\te}(x)=\max\{Mx,0\}$ respectively. It is recommended to have in mind $\p_{\te}(x)=\max\{Mx,
Q\}$ to get the first approximation of what the hypotheses mean. Notice that for the max recursion
(H2) is not satisfied on $\R $, but only on $[0, \8)\supseteq\supp\nu$.

In one dimensional case condition (H2) has a very natural geometrical interpretation, namely it can
be written in an equivalent form $M_{\te}x-|N_{\te}|\leq\p_{\te}(x)\leq M_{\te}x+|N_{\te}|$. It
means that the graph of $\p_{\te}(x)$'s lies between the graphs of $M_{\te}x-|N_{\te}|$ and
$M_{\te}x+|N_{\te}|$ for every $x\in\supp\nu$. This allows us to think that the recursion is, in a
sense, close to the affine recursion.

For simplicity we write $X$ instead of $X_{\te}$.
\begin{ass}[\textbf{Moments condition for the heavy tail}]\label{heavytail}
Let $\ka(s)=\E|M|^s$ for $s\in[0, s_{\8})$, where $s_{\8}=\sup\{s\in\R_{+}: \ka(s)<\8\}$. Let $\bm
$ be the law of $M$.
\begin{enumerate}
\item[(H3)] $G$ is the smallest closed semigroup generated by the support of $\bm$ i.e.
$G=\overline{\langle\supp\bm\rangle}$.

\item[(H4)] The conditional law of $\log|M|$, given $M\not=0$ is non arithmetic.

\item[(H5)] $M$ satisfies Cram\'{e}r
condition with exponent $\al>0$, i.e. there exists $\al\in(0, s_{\8})$  such that
$\ka(\al)=\E(|M|^{\al})=1$.
\item[(H6)] Moreover, $\E(|M|^{\al}|\log|M||)<\8$.
\item[(H7)] For the random variable $N$ defined in (H2) we have $\E(|N|^{\al})<\8$.
\end{enumerate}
\end{ass}

Conditions (H4)--(H7) are natural in this context, see \cite{BK, BDG, BDGHU, G, Gre, Gri1, Gri2,
GL1, GL2, K} and \cite{V}. Now we are ready to formulate the main result.
\begin{thm}\label{HTthm}
Assume that $\p_{\te}$ satisfies assumptions \ref{shapefun} and \ref{heavytail} for every
$\te\in\Te$. Then there is a unique stationary solution $S$ of (\ref{rec}) with law $\nu$, and
there is a unique Radon measure $\La$ on $\R^d\setminus\{0\}$ such that
\begin{align}\label{HTthm1}
\lim_{g\in G, |g|\rightarrow 0}|g|^{-\al}\E f(gS)=\La(f).
\end{align}
The convergence is valid for every bounded continuous function $f$
that vanishes in a neighbourhood of zero. Furthermore the
recursion defined in (\ref{rec}) has a heavy tail
\begin{align}\label{HTthm6}
\lim_{t\rightarrow\8}t^{\al}\P(\{|S|>t\})=\frac{1}{\al
m_\al}\E\left(|\p(S)|^{\al}-|MS|^{\al}\right),
\end{align}
where $m_{\al}=\E(|M|^{\al}\log|M|)>0$. If additionally the support of $\nu$ is unbounded, and one
of the following condition is satisfied
\begin{align}
\label{HTthm7}&s_{\8}<\8 \ \ \ \ \mbox{and}\ \ \ \ \ \lim_{s\rightarrow s_{\8}}\frac{\E(|N|^s)}{\ka(s)}=0,\\
\label{HTthm8}&s_{\8}=\8 \ \ \ \ \mbox{and}\ \ \ \ \ \lim_{s\rightarrow
\8}\left(\frac{\E(|N|^s)}{\ka(s)}\right)^{\frac{1}{s}}<\8,
\end{align}
then the measure $\La$ is nonzero.
\end{thm}
\begin{rem}\label{rem}
Contrary to Theorems \ref{limitthm} and \ref{support} the assumption that $\p_{\te}$'s are
Lipschitz is not necessary for Theorem \ref{HTthm}. The same conclusion holds if
$\p_{\te}:\R^d\mapsto\R^d$ is continuous for every $\te\in\Te$, and the map
$\Te\ni\te\mapsto\p_{\te}(x)\in\R^d$ is continuous for every $x\in\R^d$, \ref{shapefun} and
\ref{heavytail} are satisfied and $S=\lim_{n\rightarrow\8}\p_{\te_1}\circ
\p_{\te_2}\circ\ldots\circ \p_{\te_n}(x)$ exists a.s. and does not depend on $x\in\R^d$.

In view of Letac's principle \cite{L} the random variable $S$ with law $\nu$ is a unique stationary
solution of the recursion (\ref{rec}).
\end{rem}

Theorem \ref{HTthm} on one hand generalizes Theorem 1.6 of \cite{BDGHU} for multidimensional affine
recursions and on the other, the results of Goldie \cite{G} for a family of one--dimensional
recursions modeled on $ax+b$. (H4)--(H6) were already assumed by Goldie. (H3) was introduced in
\cite{BDGHU} and the whole proof is based on it. (H1)--(H2) say that asymptotically (\ref{rec})
looks like an affine recursion and it allows us to use the methods of \cite{BDGHU}.

\subsection{Limit theorem for Birhhoff sums} \label{birr} Now we introduce conditions necessary
to obtain convergence in law of appropriately normalized sums $S_n^x=\sum_{k=1}^{n}X_{k}^{x}$ to an
$\al$ -- stable distribution.
\begin{ass}[\textbf{For the limit theorem}]\label{limitass}\ \ \ \ \ \ \ \ \ \ \ \ \ \ \
\begin{enumerate}
\item[(L1)] For every $\te\in\Te$, $L_{\te}\leq|M_\te |$.
\item[(L2)] For every $\te\in\Te$,
there is a random variable $Q_{\te}$, such that $\bp_{\te}$ satisfies a smoothness condition with
respect to $t>0$, i.e. $|\p_{\te, t}(x)-\bp_{\te}(x)|\leq t\left|Q_{\te}\right|$, for every
$x\in\R^d$.
\item[(L3)] For the random variable $Q$ we have $\E(|Q|^{\al})<\8$.
\end{enumerate}
\end{ass}

Clearly, if $\bp_{\te}(x)=M_{\te}x$ for every $x\in \supp\nu$, then (L2) implies (H1), and (L2)
together with (L3) imply (H2) and (H7). Now we are ready to formulate the limit theorem.
\begin{thm}\label{limitthm}
Assume that $\p_{\te}$ satisfies assumptions \ref{shapefun}, \ref{heavytail} and \ref{limitass} for
every $\te\in\Te$. We define $S_n^x=\sum_{k=1}^{n}X_{k}^{x}$ for $n\in\N$. Let
$h_v(x)=\E\left(e^{i\left\langle v, \sum_{k=1}^{\8}\bp_{\te_k}\circ\ldots\circ
\bp_{\te_1}(x)\right\rangle}\right)$ for $x\in\R^d$, where $\bp_{\te_k}$'s were defined in (H1) of
assumption \ref{shapefun}, and let $\nu$ be the stationary measure for the recursion (\ref{rec}).
\begin{itemize}
\item If $\al\in(0, 1)\cup(1, 2)$, then there is a sequence $d_n=d_n(\al)$ and a function
$C_{\al}:\Ss^{d-1}\mapsto\C$ such that the random variables
$n^{-\frac{1}{\al}}\left(S_n^x-d_n\right)$ converge in law to the $\al$--stable random variable
with characteristic function
\begin{align*}
\ups_{\al}(tv)=\exp(t^{\al}C_{\al}(v)),\ \ \ \mbox{for $t>0$ and
$v\in\Ss^{d-1}$.}
\end{align*}
\item If $\al=1$, then there are functions $\xi, \tau:(0, \8)\mapsto\R$ and
$C_{1}:\Ss^{d-1}\mapsto\C$ such that the random variables $n^{-1}S_n^x-n\xi(n^{-1})$ converge in
law to the random variable with characteristic function
\begin{align*}
\ups_{1}(tv)=\exp(tC_{1}(v)+it\langle v, \tau(t)\rangle),\ \ \
\mbox{for $t>0$ and $v\in\Ss^{d-1}$.}
\end{align*}
\item If $\al=2$, then there is a sequence $d_n=d_n(2)$ and a function
$C_{2}:\Ss^{d-1}\mapsto\R$ such that the random variables $\left(n\log
n\right)^{-\frac{1}{2}}\left(S_n^x-d_n\right)$ converge in law to the random variable with
characteristic function
\begin{align*}
\ups_{2}(tv)=\exp(t^2C_{2}(v)),\ \ \ \mbox{for $t>0$ and
$v\in\Ss^{d-1}$.}
\end{align*}
\end{itemize}
If $\al\in(0, 1)$, then $d_n=0$, and if $\al\in(1, 2]$, then
$d_n=nm$, where $m=\int_{\R^d}x\nu(dx)$. In all the above cases
the function $C_{\al}$ depends on the function $h_v$ and the
measure $\La$ defined in Theorem \ref{HTthm}. Moreover,
$C_{\al}(tv)=t^{\al}C_{\al}(v)$ for every $t>0$, $v\in\Ss^{d-1}$
and $\al\in(0, 1)\cup(1, 2]$. If

$\supp\La$ spans $\R^d$ as a linear space, then
$\Re C_{\al}(v)<0$ for every $v\in\Ss^{d-1}$.
\end{thm}

The proof of the above theorem will be based on the spectral method that was initiated by Nagaev in
\cite{N} and then used and improved by many authors (see \cite{BG, BDG, Gou, GL1, HH1, HH2, HP} and
the references given there). The method is based on quasi--compactness of transition operators
$Pf(x)=\E\left(f(\p(x))\right)=\int_{\Te}f(\p_{\te}(x))\mu(d\te)$ on appropriate function spaces
(see \cite{BDG, GL1, HH1, HH2, HP}). They are perturbed by adding Fourier characters.

The standard use of the perturbation theory requires exponential moments of $\mu $, but there is
some development towards $\mu$'s with polynomial moments \cite{HH1} and their improvements
\cite{HP}, or even fractional moments \cite{GL1}, and \cite{BDG}. They are based on a theorem of
Keller and Liverani [24] (we refer also to [29] for an improvement of [24]). It says that the
spectral properties of the operator $P$ can be approximated by those of its Fourier perturbations
\begin{equation}\label{pert}
P_{t,v}f(x)=\E\left(e^{i\langle tv,\p(x)\rangle}f(\p(x))\right)=\int_{\Te}e^{i\langle
tv,\p_{\te}(x)\rangle}f(\p_{\te}(x))\mu(d\te),
\end{equation}
(with convention that $P_{0,v}=P$). Indeed,
\begin{equation}\label{pert1}
P_{t,v}=k_v(t)\Pi _{P, t} + Q_{P, t},
\end{equation}
where $\lim_{t\to 0}k_v(t)=1$,  $\Pi _{P, t}$ is a projection on a one dimensional subspace and the
spectral radii of $Q_{P, t}$ are smaller than $\varrho <1$, when $t\leq t_0$.
 To obtain Theorem \ref{limitthm} we need to expand the dominant
 eigenvalue $k_v(t)$ at $0$.

When $\al\in(0, 2]$, $k_v(t)$ is neither analytic nor differentiable, hence their asymptotics at
zero is much harder to obtain. The method used in \cite{BDG} does not work here and so we propose
another approach which is applicable to general Lipschitz models (see section 6).

\section{Examples}
The following examples will help the reader to understand the
meaning of the assumptions formulated in the introduction as well
as to feel the breadth of the method.

\subsection{An affine recursion}
Let $G=\R_{+}^{\ast}\times O(\R^d)$ and take the sequence of i.i.d. random pairs $(A_n,
B_n)_{n\in\N}\subseteq \Te=G\times\R^d$ with the same law $\mu$ on $\Te$ and define the affine map
$\p_n(x)=A_nx+B_n$, where $x\in\R^d$. This example was also widely considered in the context of
discrete subgroups of $\R_{+}^{\ast}$ see \cite{BDGHU} and \cite{BDG}.

\subsection{An extremal recursion}
Let $G=\R_{+}^{\ast}$ and $\Te=G\times\R$. We consider the sequence of i.i.d. pairs $(A_n,
B_n)_{n\in\N}\subseteq\Te$ with the same law $\mu $ on $\Te$. Let $\p_n(x)=\max\{A_nx, B_n\}$,
where $x\in\R$. Assume that $\P(B>0)>0$. Then
\begin{itemize}
\item $\lim_{t\rightarrow0}\p_{n,t}(x)=\bp_n(x)$, where $\bp_n(x)=\max\{A_nx,0\}$ and $M_n=A_n$.
\item The stationary solution $S$ with law $\nu$ is given by the explicit formula,
\begin{align*}
S=\max_{1\le k<\8}\{A_1A_2\cdot\ldots\cdot A_{k-1}B_k\},
\end{align*}
where $A_0=1$ a.s. \cite{G}.
\item $\P(B>0)>0$ implies that $\supp\nu\subseteq[0, \8)$ and $\supp\nu$ is unbounded.
\item In order to check cancellation condition (H2) notice that $S\geq 0$ a.s and for $x>0$
\begin{align*}
|\p_{n, t}(x)-A_nx|&=|\max\{A_nx,
B_n\}-A_nx|\mathbf{1}_{\{A_nx<B_n\}}\\
&\leq\left(|B_n|+|A_nx|\right)\mathbf{1}_{\{A_nx<B_n\}}\leq2|B_n|,
\end{align*}
so (H2) is fulfilled with $|N_n|=2|B_n|$. Moreover, we assume
(H4)--(H7) for $M_n=A_n$ and $N_n=2B_n$
\item Notice, that $|\p_{n, t}(x)-\bp_n(x)|=|\max\{A_nx, tB_n\}-\max\{A_nx, 0\}|\leq|t||B_n|$,
\end{itemize}
so (L2) is satisfied with $|Q_n|=|B_n|$.

 We deal similarly with next example.

\subsection{A model due to Letac} Let $G$ be as above and take the sequence of i.i.d. random triples
$(A_n, B_n, C_n)_{n\in\N}\subseteq\Te=G\times\R_{+}\times\R_{+}$ with the same law $\mu$ on $\Te$.
Consider the map $\p_n(x)=A_n\max\{x, B_n\}+C_n$, where $x\in\R$. If $C\geq0$ a.s. and
$\P(B>0)+\P(C>0)>0$, then the support of the stationary measure $\nu$ is unbounded \cite{G}. The
others assumptions are also satisfied.
\subsection{Another example} Take the sequence of i.i.d. random triples $(A_n, B_n,
C_n)_{n\in\N}$ $\subseteq\Te=\R_{+}^{\ast}\times\R_{+}\times\R_{+}$ with the same law $\mu$ on
$\Te$, such that $B_n^2-4A_nC_n<0$. Consider the map $\p_n(x)=\sqrt{A_nx^2+B_nx+C_n}$, where
$x\in\R$. If $\P(B>0)+\P(C>0)>0$, then the support of the stationary measure $\nu$ is unbounded
\cite{G}. Notice that $\p_n(x)$ can be written in the equivalent form
$\p_n(x)=\sqrt{A_n(x+U_n)^2+V_n}$, where $U_n=\frac{B_n}{2A_n}$ and $V_n=C_n-\frac{B_n^2}{4A_n}>0$.
Now we can easily verify that $\p_n(x)$ is Lipschitz and (L2) is satisfied. Indeed,
\begin{align*}
    \frac{|\p_n(x)-\p_n(y)|}{|x-y|}&=\frac{A_n|(x+U_n)^2-(y+U_n)^2|}{|x-y|\left(\sqrt{A_n(x+U_n)^2+V_n}+\sqrt{A_n(y+U_n)^2+V_n}\right)}\le\sqrt{A_n}.
\end{align*}
Next observe that $\bp_n(x)=\sqrt{A_n}|x|$, and
\begin{align*}
    |\p_{n,
    t}(x)-\bp_n(x)|&=\frac{\left|A_n(x+tU_n)^2+t^2V_n-A_nx^2\right|}{\sqrt{A_n(x+tU_n)^2+t^2V_n}+\sqrt{A_nx^2}}
    \le\frac{2tA_nU_n|x|}{\sqrt{A_n}|x|}
    +\frac{t^2A_nU_n^2+t^2V_n}{t\sqrt{V_n}}\\
    &\le t\left(\frac{B_n}{\sqrt{A_n}}+\frac{C_n}{\sqrt{V_n}}\right),
\end{align*}
this shows that (L2) is fulfilled.\\

For the above examples statements \ref{HTthm}, \ref{limitthm} and \ref{support} apply
straightforwardly.
\subsection{An autoregressive process with ARCH(1) errors} Now we consider an example described by
Borkovec and Kl\"{u}ppelberg in \cite{BK}. For $x\in\R$, let  $\p(x)=\left|\ga |x|+\sqrt{\be+\la
x^2}A\right|$, where $\ga\geq0, \be>0, \la>0$ are constants and $A$ is a symmetric random variable
with continuous Lebesgue density $p$, finite second moment and with the support equal the whole of
$\R$. (see section 2. in \cite{BK} for more details). Now consider the sequence
$(\p_n(x))_{n\in\N}$ of i.i.d. copies of $\p(x)$ and observe that
\begin{itemize}
\item $\lim_{t\rightarrow0}\p_{n,t}(x)=\bp_n(x)$, where $\bp_n(x)=M_n|x|$ and $M_n=\left|\ga+\sqrt{\la}A_n\right|$.
\item $|\p_{n,t}(x)-M_n|x||=\left|\left|\ga |x|+\sqrt{\be t^2+\la x^2}A_n\right|-\left|\ga+\sqrt{\la}A_n\right||x|\right|$\\
$\leq|t|\sqrt{\be}|A_n|$,
\end{itemize}
so (L2) holds with $|Q_n|=\sqrt{\be}|A_n|$. Notice that (H2) holds for every $x\in[0, \8)$ with
$|N_n|=\sqrt{\be}|A_n|$. In \cite{BK} the authors showed that it is possible to choose parameters
$\ga\geq0, \be>0, \la>0$ such that $\E\left(\log M_n\right)<0$ and $\E\left(M_n^{\al}\right)=1$ for
some $0<\al\leq2$. Observe that $\P\left(\left\{M_n\in\R_{+}^{\ast}\right\}\right)=1$. We are not
able to verify conditions  (\ref{HTthm7}) and (\ref{HTthm8}) to conclud that $\Lambda $ is not
zero, but this property follows from \cite{BK} and so Theorem \ref{limitthm} applies.
\section{Stationary measure}
\subsection{Support of the stationary measure} Let $\Cc(\R^d)$ be the set of continuous functions on
$\R^d$ and $\Cc_b(\R^d)$ be the set of bounded and continuous
functions on $\R^d$. Recall that unless otherwise stated we assume
(as in Introduction) that for every $\te\in\Te$,
$\p_{\te}:\R^{d}\mapsto \R^{d}$ is a Lipschitz map with the
Lipschitz constant $L_{\te}<\8$.

Let $\L_{\Te}^{\mu}=\overline{\{\p_{\te_1}\circ\ldots\circ\p_{\te_n}(\ \cdot\ ): \forall_{n\in\N}\
\forall_{1\leq i\leq n}\ \te_i\in\supp\mu\}}$ i.e. $\L_{\Te}^{\mu}$ is the closed semigroup
generated by the maps $\p_{\te}$, where $\te\in\supp\mu$. Given $\p_{\te}$ with $L_{\theta}<1$, let
$\p_{\te}^{\bullet}$ be the unique fixed point of $\psi _{\theta}$. Then we can formulate the main
Theorem of this section:
\begin{thm}\label{support}
Assume that $\int_{\Te}\log (L_{{\te}})\mu(d\te)<0$,
$\int_{\Te}|\log(L_{{\te}})|+\log^+(|\p_{\te}(x_0)|)\mu(d\te)<\8$
for some $x_0\in\R^d$ and $\Te\ni\te\mapsto\p_{\te}(x)\in\R^d$ is
continuous for every $x\in\R^d$. If
$\S=\{\p_{\te}^{\bullet}\in\R^d:
\p_{\te}(\p_{\te}^{\bullet})=\p_{\te}^{\bullet},\ \mbox{where}\
\p_{\te}\in\L_{\Te}^{\mu}\ \mbox{and}\ L_{\te}<1\}\subseteq\R^d$,
then $\supp\nu=\overline{\S}$, where $\nu$ is the law of the
stationary solution $S$ for the recursion (\ref{rec}).
\end{thm}
Theorem \ref{support} generalizes similar theorems for affine random walks, see \cite{BDGHU} and
\cite{Gui} for more details. Notice that conditions $\int_{\Te}\log (L_{{\te}})\mu(d\te)<0$ and
$\int_{\Te}|\log(L_{{\te}})|+\log^+(|\p_{\te}(x_0)|)\mu(d\te)<\8$ for some $x_0\in\R^d$, in view of
\cite{DF} (see also \cite{L}), give the existence of the stationary measure $\nu$ for the recursion
(\ref{rec}).

Before proving Theorem \ref{support}, we need two lemmas. Given, $\p_{\te}$ with $L_{\te}<1$, the
Banach fixed point theorem implies the existence of a unique fixed point
$\p_{\te}^{\bullet}\in\R^d$ of the map $\p_{\te}$. Moreover, for every $x\in\R^d$
\begin{align}
\label{banach1}\lim_{n\rightarrow\8}\p_{\te}^{n}(x)=\p_{\te}^{\bullet}.
\end{align}

\begin{lem}\label{banach2}
Assume that for the map $\p_{\te}$ we have $L_{\te}<1$ and $\p\in\L_{\Te}^{\mu}$. Then there exists
\begin{equation}\label{fix}
\lim_{n\rightarrow\8}(\p\circ\p_{\te}^n)^{\bullet}=\p({\p_{\te}^{\bullet}}),
\end{equation}
where $(\p\circ\p_{\te}^n)^{\bullet}\in\R^d$ is the fixed point of the map $\p\circ\p_{\te}^n$, for
$n\in\N$.
\end{lem}
\begin{proof} Notice that for $n$ sufficiently large
$\p\p_{\te}^n=\p\circ\p_{\te}^n$ is contracting. Fix
$\varepsilon>0$, then there exist $N_{\varepsilon}\in\N$ such that
$\frac{L_{\p}L_{\te}^n}{1-L_{\p}L_{\te}^n}<\varepsilon$ for all
$n\geq N_{\varepsilon}$, where $L_{\p}$ is the Lipschitz constant
associated to $\p$. For every $m\in\N$ we have
\begin{align*}
\left|(\p\p_{\te}^n)^m(\p_{\te}^{\bullet})-\p(\p_{\te}^{\bullet})\right|\leq
\left(\sum_{k=1}^{\8}(L_{\p}L_{\te}^n)^{k}\right)\cdot\left|\p(\p_{\te}^{\bullet})-\p_{\te}^{\bullet}\right|
=&\frac{L_{\p}L_{\te}^n\cdot\left|\p(\p_{\te}^{\bullet})-\p_{\te}^{\bullet}\right|}{1-L_{\p}L_{\te}^n}.
\end{align*}
By (\ref{banach1}) we can find $m\in\N$ such that
$\left|(\p\p_{\te}^n)^{\bullet}-(\p\p_{\te}^n)^m(\p_{\te}^{\bullet})\right|<\varepsilon$. Then
\begin{align*}
\left|(\p\p_{\te}^n)^{\bullet}-\p({\p_{\te}^{\bullet}})\right|&\leq\left|(\p\p_{\te}^n)^{\bullet}-(\p\p_{\te}^n)^m(\p_{\te}^{\bullet})\right|
+\left|(\p\p_{\te}^n)^m(\p_{\te}^{\bullet})-\p({\p_{\te}^{\bullet}})\right|\\
&\leq\varepsilon+\left|\p(\p_{\te}^{\bullet})-\p_{\te}^{\bullet}\right|\cdot\frac{L_{\p}L_{\te}^n}{1-L_{\p}L_{\te}^n}
\leq\varepsilon\left(1+\left|\p(\p_{\te}^{\bullet})-\p_{\te}^{\bullet}\right|\right),
\end{align*}
for all $n\geq N_{\varepsilon}$. Since $\eps$ is arbitrary, \eqref{fix} is established.
\end{proof}

\begin{lem}\label{banach3}
If $\p_{\te}:\R^d\mapsto\R^d$ is continuous for every $\te\in\Te$ (not necessarily Lipschitz) and
$\Te\ni\te\mapsto\p_{\te}(x)\in\R^d$ is continuous for every $x\in\R^d$, then for every
$\te\in\supp\mu$
\begin{align*}
\p_{\te}[\supp\nu]\subseteq\supp\nu,
\end{align*}
where the measure $\nu$ is $\mu$ stationary i.e.
$\int_{\R^d}\int_{\Te}f(\p_{\te}(x))\mu(d\te)\nu(dx)=\int_{\R^d}f(x)\nu(dx)$ for any
$f\in\Cc_b(\R^d)$.
\end{lem}
\begin{proof} Suppose for contradiction that $\p_{\te}(\supp\nu) \varsubsetneq \supp\nu.$ Then for some
$\te_0\in\supp\mu$ and $x_0\in\supp\nu$, there exists an open neighborhood $U$ of $\p_{\te_0}(x_0)$
such that $U\cap\supp\nu=\emptyset$. Notice, that $\mu(\{\te\in\Te:
\int_{\R^d}\mathbf{1}_U(\p_{\te}(x))\nu(dx)>0\})=0$, since the measure $\nu$ is $\mu$ stationary.
By the assumptions $\{\te\in\Te: \int_{\R^d}\mathbf{1}_U(\p_{\te}(x))\nu(dx)>0\}$ is an open subset
of $\Te$. $\p_{\te_0}^{-1}[U]$ is an open neighborhood of $x_0\in\supp\nu$, so
$\te_0\in\{\te\in\Te: \int_{\R^d}\mathbf{1}_U(\p_{\te}(x))\nu(dx)>0\}$, but this contradicts
$\te_0\in\supp\mu.$
\end{proof}
\begin{proof}[Proof of Theorem \ref{support}] It is a consequence of Lemma \ref{banach2} and
\ref{banach3}. Compare also with Lemma 2.7 in \cite{BDGHU}.
\end{proof}
\subsection{Simple properties of recursions and their stationary measures}
\begin{lem}\label{series} Assume that $Y_{n, t}^x=\p_{\te_1, t}\circ \p_{\te_2, t}\circ\ldots\circ \p_{\te_n, t}(x)$ for any
$n\in\N$ and $t>0$. Then,
\begin{align}
\label{series1}|Y_{n,t}^x-Y_{n,t}^y|&\leq \prod_{i=1}^{n} L_{{\te_i}}|x-y|,\\
\label{series2}|Y_{n,t}^x-Y_{n+m,t}^x|&\leq \prod_{i=1}^{n} L_{{\te_i}}|x-\p_{\te_{n+1},t}\circ\ldots\circ\p_{\te_{n+m},t}(x)|,\\
\label{series3}|x-\p_{\te_{n+1},t}\circ\ldots\circ\p_{\te_{n+m},t}(x)|&\leq\sum_{k=1}^{m}\left(\prod_{i=n+1}^{n+k-1}
L_{{\te_i}}\right)|x-\p_{\te_{n+k},t}(x)|,
\end{align}
for any $x, y\in\R^d$ and $m, n\in\N$.
\end{lem}
\begin{proof}
It is easy to see that $|\p_{\te,t}(x)-\p_{\te,t}(y)|=|t\p_{\te}(t^{-1}x)-t\p_{\te}(t^{-1}y)|\leq
L_{\te}|x-y|$ for any $x, y\in\R^d$, so (\ref{series1}), (\ref{series2}) and (\ref{series3}) follow
by induction.
\end{proof}
The next Lemma is obvious in view of what has just been
established.
\begin{lem}\label{bound3}
Under the assumptions of the previous Lemma,  if (H2), (H5), (H7) and (L1) are satisfied, then for
every $\be\in(0, \al)$, $x\in\R^d$,
\begin{align*}
\sup_{n\in\N}\left(\E|X_{n,t}^{x}|^{\be}\right)^{\frac{1}{\be}}=\sup_{n\in\N}\left(\E|Y_{n,t}^{x}|^{\be}\right)^{\frac{1}{\be}}<\8,
\end{align*}
where $X_{n, t}^x=\p_{\te_n, t}\circ \p_{\te_{n-1},
t}\circ\ldots\circ \p_{\te_1, t}(x)$ for any $n\in\N$ and $t>0$.
In particular, we obtain
$\left(\E|S|^{\be}\right)^{\frac{1}{\be}}<\8$ for every $\be\in(0,
\al)$, where $S$ is the stationary solution of (\ref{rec}).
\end{lem}

\section{The tail measure}
This section deals with a heavy tail phenomenon for Lipschitz recursions satisfying assumptions 1.7
modeled on analogous hypotheses for matrix recursions \eqref{affine}. (H1) and (H2) say that
recursion \eqref{rec} is in a sense close to the affine recursion with the linear part $M\in \R
^*_+\times K$. This allows us to use techniques of \cite{BDGHU}, in particular a generalized
renewal theorem.

Conditions \ref{heavytail} are typical for considerations of this type and they decide of
asymptotic behavior of stationary measure; especially condition (H5) is crucial. Goldie and
Gr\"{u}bel \cite{GG} showed that $\P(\{S>t\})$ can decay exponentially fast to zero if (H5) is not
satisfied.

A closed subgroup of $\R_{+}^{\ast}\times O(\R ^d)$ containing
$\R_{+}^{\ast}$ is necessarily $G=\R_{+}^{\ast}\times K$, where
$K$ is a closed subgroup of the orthogonal group $O(\R^d)$, see
e.g. Appendix C in \cite{BDGHU} and Appendix A in \cite{BDG}. Let
$\frac{dr}{r}$ be the Haar measure of $\R_{+}^{\ast}$ and let
$\rho$ be the Haar measure of $K$ such that $\rho(K)=1$. Any
element $g\in\R_{+}^{\ast}\times K$ can be uniquely written as
$g=rk$, where $r\in\R_{+}^{\ast}$ and $k\in K$, and so the Haar
measure $\la$ on $\R_{+}^{\ast}\times K$ is
$\int_{G}f(g)\la(dg)=\int_{\R_{+}^{\ast}}\int_{K}f(rk)\rho(dk)\frac{dr}{r}$.
Clearly, $G$ is unimodular.

Define convolution of a function $f$ with a measure $\mu$ on the
group $G$ as
\begin{align*}
f\ast\mu(g)=\int_{G}f(gh)\mu(dh).
\end{align*}

Given $f\in \Cc_b(\R^d)$, let
\begin{align*}
\bf(g)=\E(f(gS)),\ \ \ \mbox{and}\ \ \ \chi_f(g)=\bf(g)-\bf\ast\bm(g).
\end{align*}
The functions $\bm$ and $\chi_f$ are bounded and continuous. We
are going to express the function $\bf$ in the terms of the
potential $U=\sum_{k=0}^{\8}\bm^{\ast k}$. Notice that for any
$n\in\N\cup\{0\}$
\begin{align*}
\E(f(gM_{\te_1}M_{\te_2}\cdot\ldots\cdot M_{\te_n}S))=\int_{G}\E(f(ghS))\bm^{\ast
n}(dh)=\bf\ast\bm^{\ast n}(g).
\end{align*}
\begin{rem}\label{rem2}
Conditions (H5), (H6) imply that the function $\ka(s)=\E(|M|^s)$ is well defined on $[0,\alpha ]$
and $\ka(0)=\ka(\al)=1$. Since $\ka $ is convex, we have
\begin{align*}
\E\left(\log (|M|)\right) <0,\ \ \ \mbox{and}\ \ \ m_{\al}=\E(|M|^{\al}\log|M|)>0.
\end{align*}
For more details we refer to \cite{G}.
\end{rem}
Let $\bm_{\al}(dg)=|g|^{\al}\bm(dg)$. In view of Remark \ref{rem2}
$\bm _{\al}$ is a probability measure with positive mean and
$\bm_{\al}^{\ast n}(dg)=|g|^{\al}\bm^{\ast n}(dg)$ for all $n\in
\N$. Let $U_{\al}=\sum_{k=0}^{\8}\bm_{\al}^{\ast k}$ be the
potential kernel built out of the measure $\bm_{\al}$.

The aim of this section is to prove Theorem \ref{HTt} which implies Theorem \ref{HTthm}. Given any
Radon measure $\La$ on $\R^d\setminus\{0\}$, let define
\begin{align}
\nonumber\F_{\La}=\{f:\R^d\mapsto\R: f\ \mbox{is measurable function such that $\La(\Dis(f))=0$, and}\\
\label{HTthm2}\sup_{x\in\R^d}|x|^{-\al}|\log|x||^{1+\eps}|f(x)|<\8\ \mbox{for some $\eps>0$}\},
\end{align}
where $\Dis(f)$ is the set of all discontinuities of function $f$.
\begin{thm}\label{HTt}

Suppose that 1.6 and 1.7 are satisfied. Then  there is a unique stationary solution $S$ of
(\ref{rec}) with the law $\nu$, and there is a unique Radon measure $\La$ on $\R^d\setminus\{0\}$
such that
\begin{align}\label{HTthm11}
\lim_{|g|\rightarrow 0}|g|^{-\al}\E f(gS)=\lim_{|g|\rightarrow
0}|g|^{-\al}\int_{\R^d}f(gx)\nu(dx)=\int_{\R^d\setminus\{0\}}f(x)\La(dx),
\end{align}
for every function $f\in\F_{\La}$. The measure $\La $ is homogeneous with degree $\alpha$ i.e.
$\int_{\R^d}f(gx)\La(dx)=|g|^{\al}\La(f)$ for every $g\in G$. There exists a measure $\si_{\La}$ on
$\Ss^{d-1}$ such that $\La$ has the polar decomposition
\begin{align}\label{HTthm4}
\int_{\R^d\setminus\{0\}}f(x)\La(dx)=\int_{0}^{\8}\int_{\Ss^{d-1}}f(rx)\si_{\La}(dx)\frac{dr}{r^{\al+1}},
\end{align}
where $\si_{\La}(\Ss^{d-1})=\frac{1}{m_\al}\E\left(|\p(S)|^{\al}-|MS|^{\al}\right)$ and
$m_{\al}=\E(|M|^{\al}\log|M|)\in(0, \8)$. Furthermore, recursion defined in (\ref{rec}) has a heavy
tail
\begin{align}\label{HTthm66}
\lim_{t\rightarrow\8}t^{\al}\P(\{|S|>t\})=\frac{1}{\al
m_\al}\E\left(|\p(S)|^{\al}-|MS|^{\al}\right).
\end{align}
If additionally the support of $\nu$ is unbounded, and one of the following condition is satisfied
\begin{align}
\label{HTthm77}&s_{\8}<\8 \ \ \ \ \mbox{and}\ \ \ \ \ \lim_{s\rightarrow s_{\8}}\frac{\E(|N|^s)}{\ka(s)}=0,\\
\label{HTthm88}&s_{\8}=\8 \ \ \ \ \mbox{and}\ \ \ \ \ \lim_{s\rightarrow
\8}\left(\frac{\E(|N|^s)}{\ka(s)}\right)^{\frac{1}{s}}<\8,
\end{align}
then the measures $\La$ and $\si_{\La}$ are nonzero.
\end{thm}

We divide the proof into three steps. Step 1. (Existence of the tail measure $\La$) and Step 3.
(Nontriviality of the tail measure $\La$) go along the same lines as the Main Theorem 1.6. in
\cite{BDGHU} so we give only outlines of proofs. The proof of the existence of a polar
decomposition for the measure $\La$ is shorter here and it is given in the Step 2..

\begin{proof}
\textit{Step 1. Existence of the tail measure $\La$.} Assumptions (\ref{heavytail}) imply the
existence of the stationary solution $S$ for the recursion (\ref{rec}) with the law $\nu$ (see
\cite{DF}, \cite{L} for more details). Now, for an $\eps\in(0, 1]$, we define the set of H\"{o}lder
functions by
\begin{align*}
\H_{\eps}=\{f\in\Cc_b(\R^d): \forall_{x, y\in\R^d}\ &|f(x)-f(y)|\leq C_f|x-y|^{\eps},\\
&\mbox{and $f$ vanishes in a neighbourhood of $0$}\}.
\end{align*}
Given $ f\in \H_{\eps}$ for some $\eps\in(0, 1]$ and $\eps<\al$,
we write $\chi_{f, \al}(g)=|g|^{-\al}\chi_f(g)$. Using
cancellation condition (H2) and arguing in a similar way as in
\cite{BDGHU} (see Lemma 2.19) we obtain that the function
$\chi_{f, \al}(g)=|g|^{-\al}\chi_f(g)$ is $d\Rr i$ (\textit{direct
Riemann integrable} on $G$, definition of $d\Rr i$ functions can
be found in \cite{BDGHU}). Now we can use a renewal theorem for
closed subgroups of $\R^*_{+}\times K$, where $K$ is a metrizable
group not necessarily Abelian, (see Appendix A of \cite{BDGHU},
also \cite{Gui1} and \cite{R}). It is applied to the function
$\chi_{f, \al}(g)$, to obtain
\begin{align}\label{La1}
\lim_{|g|\rightarrow 0}|g|^{-\al}\bf(g)=\lim_{|g|\rightarrow
0}U_{\al}(\chi_{f,\al})(g)=\frac{1}{m_{\al}}\int_{G}\chi_{f,\al}(g)\la(dg).
\end{align}
The formula
\begin{align*}
\La(f)=\frac{1}{m_{\al}}\int_{G}\chi_{f,\al}(g)\la(dg)=\frac{1}{m_{\al}}\int_{G}|g|^{-\al}(\E
(f(gS)-f(gMS)))\la(dg),
\end{align*}
defines a nonnegative Radon measure on $\R^d\setminus\{0\}$, which is $\al$ homogeneous.
Convergence in (\ref{La1}) holds also for
$f\in\F_{\La}$ (compare with the proof of Theorem 2.8 in \cite{BDGHU}).\\
\noindent\textit{Step 2. Polar decomposition for the measure $\La$.} Being homogeneous $\La $ can
be nicely expressed in polar coordinates. Let $\Phi: \R^d\setminus\{0\}\mapsto(0,
\8)\times\Ss^{d-1}$ be defined as follows $\Phi(x)=\left(|x|, \frac{x}{|x|}\right)$ and its inverse
$\Phi^{-1}:(0, \8)\times\Ss^{d-1}\mapsto\R^d\setminus\{0\}$ by $\Phi^{-1}(r, z)=rz$. Next we define
the measures $\si^s$ on $\Ss^{d-1}$
\begin{align*}
\si^s(F)=s\La^s\left(\Phi^{-1}\left[[1, \8)\times F\right]\right),
\end{align*}
where $\La^s(f)=\frac{1}{m_\al}\int_{G}|g|^{-s}(\E (f(gS)-f(gMS)))\la(dg)$ for $s<\al$ and $F\in\B
or(\Ss^{d-1})$. $\B or(X)$ means the Borel $\si$--field of $X$. Fix $0<\be<\ga$ and notice that for
any $[\be, \ga)\times F\in\B or((0,\8))\otimes\B or(\Ss^{d-1})$,
\begin{align*}
\left(\La^s\circ\Phi^{-1}\right)([\be, \ga)\times F) &=\si^s(F)\int_{\be}^{\ga}\frac{dr}{r^{s+1}}.
\end{align*}
The above proves (\ref{HTthm4}) with the measure $\La^s$ instead of $\La$. Now notice that
\begin{align*}
\si^s(\Ss^{d-1})&=s\La^s\left(\Phi^{-1}\left[[1, \8)\times
\Ss^{d-1}\right]\right)=\frac{1}{m_\al}\E\left(|\p(S)|^s-|MS|^s\right).
\end{align*}
Hence (\ref{HTthm4}) holds. Furthermore,
\begin{align*}
\lim_{t\rightarrow\8}t^{\al}\P(\{|S|>t\})=\lim_{t\rightarrow\8}t^{\al}\int_{t}^{\8}\frac{dr}{r^{\al+1}}\si_{\La}(\Ss^{d-1})=\frac{1}{\al
m_\al}\E\left(|\p(S)|^{\al}-|MS|^{\al}\right),
\end{align*}
and (\ref{HTthm66}) also holds.\\
\noindent\textit{Step 3. Nontriviality of the tail measure $\La$.} In order to prove that measure
$\La$ is nontrivial in view of (\ref{HTthm4}) we have to show that $\si_{\La}\not=0$. Suppose for a
contradiction that $\si_{\La}(\Ss^{d-1})=0$. Applying the method from section 3 from \cite{BDGHU}
together with condition (H2) we obtain that the stationary solution $S$ is bounded, but it
contradicts with the fact that the support of the measure $\nu$ is unbounded and this finishes the
proof of Theorem \ref{HTt}.
\end{proof}
The example below shows that for (\ref{HTthm77}) and (\ref{HTthm88}) the hypothesis that the
support of the measure $\nu$ is unbounded is crucial. Consider $\p_n(x)=A_n\max\{x, B_n\}+C_n$ and
assume that $\P\left(\left\{A_n=\frac{1}{3}\right\}\right)=\frac{3}{4},$
$\P\left(\left\{A_n=2\right\}\right)=\frac{1}{4}$ and
$\P\left(\left\{B_n=\frac{1}{2}\right\}\right)=\P\left(\left\{C_n=-1\right\}\right)=1$. Then
$\E(\log A_n)<0$ and $\E(A_n^{\al})=1$, where $\al\approx1,851$. It is easy to see that the
stationary measure $\nu$ is supported by the set $\left\{-\frac{5}{6}, 0\right\}$ though the
function $\p_n(x)$ is unbounded.

%%%%%%%%%%%%%%%%%%%%%%%%%%%%%%%%%%%%%%%%%%%%%%%%%%%%%%%%%%%%%%%%%%%%%%%%%%%%%%%%%%%%%%%%%%%%%%%%%%%%%%%%%%%%%%%%%%%%%%%%%%%%%%%%%%%%%%%%%%%%%%%%%%%%%%%%%%%%%%%%%%%%%%%%%%%%%%%%%%%%%%%%%%%%%%%%%%%
\section{Fourier operators and their properties}
As it is mentioned in subsection \ref{birr}, for the limit Theorem \ref{limitthm}, we study the
Markov operator $P$ associated to the recursion (\ref{rec}) as well as the perturbations $P_{t,v}$
of $P$ defined in (\ref{pert}). To expand the dominant eigenvalue $k_v(t)$ defined in (\ref{pert1})
we need some information about the eigenfunctions $\Pi _{P,t}1$, $t\geq 0$. While $t$ varies, the
normalization of $\Pi _ {P,t}1$ counts. This means that although the corresponding eigenspaces are
one dimansional, the choice of a multiple of $\Pi _ {P,t}1$ is delicate. Properties (L1)--(L3)
allow us to proceed similarly as in
 \cite{BDG}, but not exactly. The major difference is related to auxiliary
operators $T_{t,v}$. They are used in \cite{BDG} to obtain an explicit expression for $\Pi _
{P,t}1$, but they are written there by the formula that does not work beyond the affine recursion.
However, a careful analysis of operators $T_{t,v}$ suggests to write them abstractly as
\begin{equation}\label{dilat}
T_{t, v}=\Dt^{-1}\circ P_{t, v}\circ\Dt,
\end{equation}
where $\Dt$ is the dilatation  $\Dt f(x)=f(tx)$. The ``abstract'' $T_{t,v}$'s
do the same job making the method applicable to a much more general context.

We start by introducing two Banach spaces $\ct(\R^d)$ and $\bl1(\R^d)$ of continuous functions
\cite{LP} (see also \cite{BDG, GL1, HH1, HH2} and \cite{HP}).
\begin{align*}
\ct=\ct(\R^d)&=\left\{f\in\Cc(\R^d): |f|_{\rho}=\sup_{x\in\R^d}\frac{|f(x)|}{(1+|x|)^{\rho}}<\8\right\},\\
\bl1=\bl1(\R^d)&=\{f\in\Cc(\R^d): \|f\|_{\rel}=|f|_{\rho}+[f]_{\el}<\8\},
\end{align*}
where
\begin{align*}
[f]_{\el}=\sup_{x\not= y}\frac{|f(x)-f(y)|}{|x-y|^{\ep}(1+|x|)^{\la}(1+|y|)^{\la}}.
\end{align*}
%is the seminorm.
\begin{rem}\label{Arcela}
If $\ep+\la<\rho$, then $[f]_{\el}<\8$ implies $|f|_{\rho}<\8$. As a simple application of
Arzel\`{a} -- Ascoli theorem %and simple properties of above norms
we obtain that the injection operator $\bl1\hookrightarrow\ct$ is compact.
\end{rem}
From now on we assume that $\p_{\te}$ satisfies \ref{shapefun}, \ref{heavytail} and \ref{limitass}
for every $\te\in\Te$. For the sake of simplicity we write $\p$ instead of $\p_{\te}$. On $\ct$ and
$\bl1$ we consider the transition operator $Pf(x)=\E\left(f(\p(x))\right)$ and its perturbations
\begin{align*}
P_{t,v}f(x)=\E\left(e^{i\langle tv,\p(x)\rangle}f(\p(x))\right),
\end{align*}
defined in (\ref{pert}), where $x\in\R^d$, $v\in\Ss^{d-1}$ and $t\in[0, 1]$. Observe that
$P_{0,v}=P$. For convenience we write $\p_t=\bp$ for $t=0$, ($\bp$ was defined in (H1) of
assumption (\ref{shapefun})). We will also use the family of Fourier operators $T_{t,v}$ defined in
(\ref{dilat}). Notice that
\begin{align*}
T_{t,v}f(x)=\left(\Dt^{-1}\circ P_{t, v}\circ\Dt\right)f(x)=\E\left(e^{i\langle
v,\p_{t}(x)\rangle}f(\p_{t}(x))\right),
\end{align*}
for $x\in\R^d$, where $t\in[0, 1]$ and $v\in\Ss^{d-1}$. For simplicity we will write
$T_{v}=T_{0,v}$. $T_{t,v}$ are perturbations of $T_v$.

Clearly, for $t>0$, $T_{t,v}$ and $P_{t,v}$, as being dilations of each other, have the same
peripherical eigenvalues $k_v(t)$, but for $t=0$, the relation between $T_v$ and $P$ is not that
close. Therefore, by considering $T_v$ and $T_{t,v}$ we obtain some extra information when $t\to
0$. In particular, the eigenfunction $h_v$ of $T_v$ with the eigenvalue $1$ plays a vital role in
approximating peripherical eigenvectors of $P_{t,v}$.

To treat both $P_{t,v}$ and $T_{t,v}$ in a unified way we write
\begin{align*}
\F_{s,t,v}f(x)=\E\left(e^{i\langle sv,\p_{t}(x)\rangle}f(\p_{t}(x))\right)=\int_{\Te}e^{i\langle
sv, \p_{\te, t}(x)\rangle}f(\p_{\te, t}(x))\mu(d\te).
\end{align*}
Notice that $\F_{s,0,v}f(x)=\E\left(e^{i\left\langle sv,
\bp(x)\right\rangle}f\left(\bp(x)\right)\right)=\int_{\Te}e^{i\left\langle sv,
\bp_{\te}(x)\right\rangle}f\left(\bp_{\te}(x)\right)\mu(d\te),$ and
$\F_{0,t,v}f(x)=\E\left(f(\p_{t}(x))\right)=\int_{\Te}f(\p_{\te,t}(x))\mu(d\te),$ for $x\in\R^d$,
where $s, t\in[0, 1]$ and $v\in\Ss^{d-1}$. Observe that, $\F_{s,1,v}=P_{s,v}$ and $\F_{1,t,
v}=T_{t,v}$.

Now by the definition (\ref{dilat}) it is easy to see, that for every $n\in\N$ and $t\in[0, 1]$,
\begin{align}\label{PT1}
P_{t, v}^{n}\circ\Dt=\Dt\circ T_{t, v}^{n}.
\end{align}
Moreover, if $f\in\ct$ is eigenfunction of operator $T_{t, v}$ with eigenvalue $k_v(t)$, then $\Dt
f$ is an eigenfunction of the operator $P_{t, v}$ with the same eigenvalue. The main result of this
section is the following

\begin{prop}\label{KL}
Assume that $0<\ep<1$, $\la>0$, $\la+2\ep<\rho=2\la$ and $2\la+\ep<\al$, then there exist
$0<\varrho<1$, $\de>0$ and $t_0>0$ such that $\varrho<1-\de$, and for every $t\in[0, t_0]$ and
every $v\in\Ss^{d-1}$
\begin{itemize}
\item $\sigma(P_{t, v})$ and $\sigma(T_{t, v})$ are contained in $\D=\{z\in\C: |z|\leq\varrho\}\cup\{z\in\C: |z-1|\leq\de\}$.
\item The sets $\sigma(P_{t, v})\cap\{z\in\C: |z-1|\leq\de\}$ and $\sigma(T_{t, v})\cap\{z\in\C:
|z-1|\leq\de\}$ consist of exactly one eigenvalue $k_v(t)$, where $\lim_{t\rightarrow0}k_v(t)=1,$
and the corresponding eigenspace is one dimensional.
\item For any $z\in\D^c$ and every $f\in\bl1$
\begin{align*}
\ \ \ \ \ \ \ \left\|(z-P_{t, v})^{-1}f\right\|_{\rel}\leq D\|f\|_{\rel},\ \mbox{and}\
\left\|(z-T_{t, v})^{-1}f\right\|_{\rel}\leq D\|f\|_{\rel},
\end{align*}
where $D>0$ is universal constant which does not depend on $t\in[0, t_0]$.
\item Moreover, we can express operators $P_{t,v}$ and $T_{t,v}$ in the
following form
\begin{align*}
P_{t,v}^n=k_v(t)^n\Pi_{P, t}+Q_{P, t}^n,\ \mbox{and}\ \ T_{t,v}^n=k_v(t)^n\Pi_{T, t}+Q_{T, t}^n,
\end{align*}
for every $n\in\N$, where $\Pi_{P, t}$ and $\Pi_{T, t}$ are projections onto the one dimensional
eigenspaces mentioned above. $Q_{P, t}$ and $Q_{T, t}$ are the complementary operators to
projections $\Pi_{P, t}$ and $\Pi_{T, t}$ respectively, such that $\Pi_{P, t}Q_{P, t}=Q_{P,
t}\Pi_{P, t}=0$ and $\Pi_{T, t}Q_{T, t}=Q_{T, t}\Pi_{T, t}=0$. Furthermore, $\|Q_{P,
t}^n\|_{\bl1}=O(\varrho^n)$ and $\|Q_{T, t}^n\|_{\bl1}=O(\varrho^n)$ for every $n\in\N$. The
operators $\Pi_{P, t}$, $\Pi_{T, t}$, $Q_{P, t}$ and $Q_{T, t}$ depend on $v\in\Ss^{d-1}$, but this
is omitted for simplicity.
\item The above operators can be expressed in the
terms of the resolvents of $P_{t, v}$ and $T_{t, v}$. Indeed, for appropriately chosen $\xi_1>0$
and $\xi_2>0$ we have
\begin{align*}
k_v(t)\Pi_{F, t}&=\frac{1}{2\pi i}\int_{|z-1|=\xi_1}z(z-F_{t, v})^{-1}dz,\\
\Pi_{F, t}&=\frac{1}{2\pi i}\int_{|z-1|=\xi_1}(z-F_{t, v})^{-1}dz,\\
Q_{F, t}&=\frac{1}{2\pi i}\int_{|z|=\xi_2}z(z-F_{t, v})^{-1}dz,
\end{align*}
where $F_{t, v}=P_{t, v}$ or $F_{t, v}=T_{t, v}$.
\end{itemize}
\end{prop}
Proposition \ref{KL} is a consequence of the perturbation theorem of Keller and Liverani \cite{KL}
see also \cite{Liv}. Before we apply their theorem we will check in a number of Lemmas that its
assumptions are satisfied.
\begin{lem}\label{KL0}
For every $n\in\N$
\begin{align}\label{per}
\F_{s,t,v}^{n}f(x)=\E\left(e^{i\left\langle sv,
S_{n,t}^x\right\rangle}f\left(X_{n,t}^{x}\right)\right),
\end{align}
where $X_{n,t}^{x}$ is defined as in Lemma \ref{bound3} and $S_{n, t}^x=\sum_{k=1}^{n}X_{k,
t}^{x}$.
\end{lem}
\begin{proof} Formula \eqref{per} is obvious, since $\{X_{n,t}^{x}\}_{n\geq1}$ is a Markov chain.
\end{proof}
\begin{rem}\label{KL1}
The formula (\ref{per}) implies that for every $0<\rho<\al$, there exists a constant $C_1>0$
independent of $s, t\in[0, 1]$ and $v\in\Ss^{d-1}$ such that for every $n\in\N$
\begin{align}\label{KL1A}
\left|\F_{s,t,v}^{n}f\right|_{\rho}\leq C_1|f|_{\rho}.
\end{align}
\end{rem}
Let denote $\Pi_n=L_{\te_1}\cdot\ldots\cdot L_{\te_n}$ for $n\in\N$ and $\Pi_0=1$. The inequality
$|e^{ix}-1|\le2|x|^{\eps}$ for $0<\eps\le1$ and $x\in\R$ will be used repeatedly.
\begin{lem}\label{KL2}
Assume that $0<\ep<1$, $\la>0$, $2\la+\ep<\al$, and $\rho=2\la$. Then there exist constants
$C_2>0$, $C_3>0$ and $0<\varrho<1$ independent of $s, t\in[0, 1]$ and $v\in\Ss^{d-1}$ such that for
every $f\in\bl1$ and $n\in\N$
\begin{align}\label{KL2A}
\left[\F_{s,t,v}^{n}f\right]_{\el}\leq C_2\varrho^n[f]_{\el}+C_3|f|_{\rho}.
\end{align}
\end{lem}
\begin{proof} By the definition of the seminorm $[\ \cdot\ ]_{\el}$ we have
\begin{align}
\label{KL2a}\F_{s,t,v}^{n}f(x)-\F_{s,t,v}^{n}f(y)&=\E\left(e^{i\left\langle sv,
S_{n,t}^x\right\rangle}\left(f\left(X_{n,t}^{x}\right)-f\left(X_{n,t}^{y}\right)\right)\right)\\
\label{KL2b}&+\E\left(\left(e^{i\left\langle sv, S_{n,t}^x\right\rangle}-e^{i\left\langle sv,
S_{n,t}^y\right\rangle}\right)f\left(X_{n,t}^{y}\right)\right).
\end{align}
To obtain (\ref{KL2A}) we have to estimate (\ref{KL2a}) and (\ref{KL2b}) separately. Indeed,
\begin{align}
\label{KL2c}&\frac{\left|\E\left(e^{i\left\langle sv,
S_{n,t}^x\right\rangle}\left(f\left(X_{n,t}^{x}\right)-f\left(X_{n,t}^{y}\right)\right)\right)\right|}{|x-y|^{\ep}(1+|x|)^{\la}(1+|y|)^{\la}}\\
\nonumber&\leq[f]_{\el}\cdot\E\left(\frac{\left|X_{n,t}^{x}-X_{n,t}^{y}\right|^{\ep}\left(1+\left|X_{n,t}^{x}\right|\right)^{\la}\left(1+\left|X_{n,t}^{y}\right|\right)^{\la}}{|x-y|^{\ep}(1+|x|)^{\la}(1+|y|)^{\la}}\right)\\
\nonumber&\leq[f]_{\el}\cdot\E\left(\frac{\Pi_n^{\ep}\left(1+\left|X_{n,t}^{0}\right|+\Pi_n|x|\right)^{\la}\left(1+\left|X_{n,t}^{0}\right|+\Pi_n|y|\right)^{\la}}{(1+|x|)^{\la}(1+|y|)^{\la}}\right)\\
\nonumber&\leq[f]_{\el}\cdot\E\left(\Pi_n^{\ep}\left(\Pi_n+\left|X_{n,t}^{0}\right|+1 \right)^{2\la}\right)\\
\nonumber&\leq3^{2\la}[f]_{\el}\cdot\left(\E\left(\Pi_n^{2\la+\ep}\right)+\E\left(\Pi_n^{\ep}\left|X_{n,t}^{0}\right|^{2\la}\right)+\E\left(\Pi_n^{\ep}\right)\right).
\end{align}
Now let $\varrho=\max\left\{\ka(\ep), \ka(2\la+\ep),
\ka^{\frac{\ep}{2\la+\ep}}(2\la+\ep)\right\}<1$. Applying the H\"{o}lder inequality to the last
expression, we obtain
\begin{align}
\label{KL2d}&3^{2\la}[f]_{\el}\cdot\left(\E\left(\Pi_n^{2\la+\ep}\right)+\E\left(\Pi_n^{\ep}\left|X_{n,t}^{0}\right|^{2\la}\right)+\E\left(\Pi_n^{\ep}\right)\right)\\
\nonumber\leq3^{2\la}[f]_{\el}\cdot&\left(\ka(2\la+\ep)^n+t^{2\la}\left(\ka^{\frac{\ep}{2\la+\ep}}(2\la+\ep)\right)^n\E\left(\left|X_{n}^{0}\right|^{2\la+\ep}\right)^{\frac{2\la}{2\la+\ep}}+\ka(\ep)^n\right)\\
\nonumber\leq3^{2\la}\varrho^n[f]_{\el}\cdot&\left(2+t^{2\la}\E\left(\left|X_{n}^{0}\right|^{2\la+\ep}\right)^{\frac{2\la}{2\la+\ep}}\right)\leq
C_2\varrho^n[f]_{\el},
\end{align}
where by Lemma \ref{bound3} the constant
 $C_2=3^{2\la}\sup_{n\in\N}\left(2+\E\left(\left|X_{n}^{0}\right|^{2\la+\ep}\right)^{\frac{2\la}{2\la+\ep}}\right)$ is finite.

In order to estimate (\ref{KL2b}) notice that we have
\begin{align*}
\left|S_{n,t}^x-S_{n,t}^y\right|\leq\sum_{k=1}^{n}\left|X_{k,t}^{x}-X_{k,t}^{y}\right|\leq\sum_{k=1}^{n}\Pi_k|x-y|\leq
B_n|x-y|,
\end{align*}
where $B_n=\sum_{k=0}^{n}\Pi_k$. Assume that $|y|\leq|x|$, then
\begin{align}
\label{KL2e}&\frac{\left|\E\left(\left(e^{i\left\langle sv,
S_{n,t}^x\right\rangle}-e^{i\left\langle sv,
S_{n,t}^y\right\rangle}\right)f\left(X_{n,t}^{y}\right)\right)\right|}{|x-y|^{\ep}(1+|x|)^{\la}(1+|y|)^{\la}}\\
\nonumber&\leq|f|_{\rho}\cdot\E\left(\frac{\left|e^{i\left\langle sv,
S_{n,t}^x-S_{n,t}^y\right\rangle}-1\right|\left(1+\left|X_{n,t}^{y}\right|\right)^{\rho}}{|x-y|^{\ep}(1+|x|)^{\la}(1+|y|)^{\la}}\right)\\
\nonumber&\leq2s^{\ep}|f|_{\rho}\cdot\E\left(\frac{B_{n}^{\ep}\left(1+\left|X_{n,t}^{0}\right|+\Pi_n\right)^{\rho}(1+|y|)^{\rho}}{(1+|y|)^{2\la}}\right)\\
\nonumber&\leq2\cdot3^{\rho}s^{\ep}|f|_{\rho}\cdot\E\left(B_{n}^{\ep}+t^{\rho}B_{n}^{\ep}\left|X_{n}^{0}\right|^{\rho}+B_{n}^{\ep}\Pi_n^{\rho}\right)\leq
C_3|f|_{\rho},
\end{align}
where the constant
$C_3=\sup_{n\in\N}2\cdot3^{\rho}\cdot\E\left(B_{n}^{\ep}+B_{n}^{\ep}\left|X_{n}^{0}\right|^{\rho}+B_{n}^{\ep}\Pi_n^{\rho}\right)$
is finite by the similar argument as in the previous case. This
completes the proof of the Lemma.
\end{proof}
\begin{lem}\label{KL3}
Assume that $0<\ep<1$, $\la>0$, $2\la+\ep<\al$, $\rho=2\la$ and $\la+2\ep<\rho$. Then there exist
finite constants $C_4>0$ and $C_5>0$ independent of $s, t\in[0, 1]$ and of $v\in\Ss^{d-1}$ such
that for every $f\in\bl1$
\begin{align}
\label{KL3A}\left|\left(\F_{s,t,v}-\F_{s,0,v}\right)f\right|_{\rho}\leq C_4t^{\ep}\|f\|_{\rel},\\
\label{KL3B}\left|\left(\F_{s,t,v}-\F_{0,t,v}\right)f\right|_{\rho}\leq C_5s^{\ep}\|f\|_{\rel}.
\end{align}
\end{lem}
Notice that this Lemma also applies to the special case when $\F_{1,0,v}=T_v$.
\begin{proof}
In order to prove (\ref{KL3A}) we write
\begin{align}
\label{KL3a}\left(\F_{s,t,v}-\F_{s,0,v}\right)f(x)&=\E\left(e^{i\langle
sv,\p_{t}(x)\rangle}\left(f(\p_{t}(x))-f\left(\bp(x)\right)\right)\right)\\
\label{KL3b}&+\E\left(\left(e^{i\langle sv,\p_{t}(x)\rangle}-e^{i\left\langle sv,
\bp(x)\right\rangle}\right)f\left(\bp(x)\right)\right).
\end{align}
Now we estimate (\ref{KL3a}) and (\ref{KL3b}) separately. By the definition of map $\bp$ we know
that $\bp(0)=0$, so $\left|\bp(x)\right|\leq|M||x|$. Then condition (L2) implies that
$|\p_t(x)|\leq t|Q|+|M||x|$ and so
\begin{align}
\label{KL3c}&\frac{\left|\E\left(e^{i\langle
sv,\p_{t}(x)\rangle}\left(f(\p_{t}(x))-f\left(\bp(x)\right)\right)\right)\right|}{(1+|x|)^{\rho}}\leq\E\left(\frac{\left|f(\p_{t}(x))-f\left(\bp(x)\right)\right|}{(1+|x|)^{\rho}}\right)\\
\nonumber&\leq[f]_{\el}\cdot\E\left(\frac{\left|\p_t(x)-\bp(x)\right|^{\ep}(1+|\p_t(x)|)^{\la}\left(1+\left|\bp(x)\right|\right)^{\la}}{(1+|x|)^{\rho}}\right)\\
\nonumber&\leq t^{\ep}[f]_{\el}\cdot\E\left(\frac{|Q|^{\ep}(1+t|Q|+|M||x|)^{2\la}}{(1+|x|)^{\rho}}\right)\\
\nonumber&\leq3^{2\la}t^{\ep}[f]_{\el}\cdot\E\left(|Q|^{\ep}+t^{2\la}|Q|^{2\la+\ep}+|Q|^{\ep}|M|^{2\la}\right)\leq
D_1t^{\ep}[f]_{\el}.
\end{align}

The quantity $D_1=3^{2\la}\cdot\E\left(|Q|^{\ep}+|Q|^{2\la+\ep}+|Q|^{\ep}|M|^{2\la}\right)$ is
finite according to H\"{o}lder's inequality, (H5) and (L3).

For (\ref{KL3b}), we have
\begin{align}
\label{KL3d}&\frac{\left|\E\left(\left(e^{i\langle sv,\p_{t}(x)\rangle}-e^{i\left\langle sv,
\bp(x)\right\rangle}\right)f\left(\bp(x)\right)\right)\right|}{(1+|x|)^{\rho}}\\
\nonumber&\leq|f|_{\rho}\cdot\E\left(\frac{\left|e^{i\left\langle
sv,\p_t(x)-\bp(x)\right\rangle}-1\right|(1+|M||x|)^{\rho}}{(1+|x|)^{\rho}}\right)\\
\nonumber&\leq2s^{\ep}t^{\ep}|f|_{\rho}\cdot\E\left(|Q|^{\ep}(1+|M|)^{\rho}\right)\\
\nonumber&\leq2^{\rho+1}s^{\ep}t^{\ep}|f|_{\rho}\cdot\E\left(|Q|^{\ep}+|Q|^{\ep}|M|^{\rho}\right)\leq
D_2t^{\ep}|f|_{\rho},
\end{align}
where the constant $D_2=2^{\rho+1}\cdot\E\left(|Q|^{\ep}+|Q|^{\ep}|M|^{\rho}\right)$ is also finite
by the H\"{o}lder inequality, (H5) and (L3). Combining (\ref{KL3c}) with (\ref{KL3d}) we obtain
(\ref{KL3A}) with $C_4=\max\{D_1, D_2\}$.

In order to prove (\ref{KL3B}) notice that
\begin{align}
\label{KL3e}&\frac{\left|\left(\F_{s,t,v}-\F_{0,t,v}\right)f(x)\right|}{(1+|x|)^{\rho}}\leq\E\left(\frac{\left|e^{i\langle
sv,\p_t(x)\rangle}f(\p_t(x))-f(\p_t(x))\right|}{(1+|x|)^{\rho}}\right)\\
\nonumber&\leq\E\left(\frac{|e^{i\langle
sv,\p_t(x)\rangle}-1||f(\p_t(x))-f(0)|}{(1+|x|)^{\rho}}\right)+\E\left(\frac{|e^{i\langle
sv,\p_t(x)\rangle}-1||f(0)|}{(1+|x|)^{\rho}}\right)\\
\nonumber&\leq2s^{\ep}\left([f]_{\el}\cdot\E\left(\frac{|\p_t(x)|^{2\ep}(1+|\p_t(x)|)^{\la}}{(1+|x|)^{\rho}}\right)+|f|_{\rho}\cdot\E\left(\frac{|\p_t(x)|^{\ep}}{(1+|x|)^{\rho}}\right)\right)\\
\nonumber&\leq2^{\la+1}s^{\ep}\|f\|_{\rel}\cdot\E\left(\frac{|\p_{t}(x)|^{2\ep}+|\p_{t}(x)|^{\la+2\ep}+|\p_{t}(x)|^{\ep}}{(1+|x|)^{\rho}}\right)\\
\nonumber&\leq C_5s^{\ep}\|f\|_{\rel},
\end{align}
where
$C_5=2^{\la+1}\cdot\E\left(\left(1+|M|+|Q|\right)^{2\ep}+\left(1+|M|+|Q|\right)^{\la+2\ep}+\left(1+|M|+|Q|\right)^{\ep}\right)$
is finite by (H5) and (L3). Hence (\ref{KL3e}) proves (\ref{KL3B}) and finally it completes the
proof of the Lemma.
\end{proof}
\begin{lem}\label{EigenvalueP}
The unique eigenvalue of modulus $1$ for operator $P$ acting on $\ct$ is $1$ and the eigenspace is
one dimensional. The corresponding projection on $\C\cdot1$ is given by the map $f\mapsto\nu(f)$.
\end{lem}
\begin{proof}
The proof can be found in section 3 of \cite{BDG}.
\end{proof}
Recall, that for every $n\in\N$,
\begin{align*}
T_v^{n}f(x)=\E\left(e^{i\left\langle v, \sum_{k=1}^{n}\bp_{k}\circ\ldots\circ
\bp_{1}(x)\right\rangle}f\left(\bp_{n}\circ\ldots\circ \bp_{1}(x)\right)\right),
\end{align*}
where $\bp_{k}(x)=\bp_{\te_k}(x)$, and $\bp_{\te_k}$'s were defined in (H1) of assumption
\ref{shapefun}. Let
\begin{align}\label{ET}
h_v(x)=\E\left(e^{i\left\langle v, \sum_{k=1}^{\8}\bp_{k}\circ\ldots\circ
\bp_{1}(x)\right\rangle}\right).
\end{align}
\begin{lem}\label{ETL}
The function $h_v$ defined in (\ref{ET}) belongs to $\bl1$, and $h_v(tx)=h_{tv}(x)$ for every
$x\in\R^d$ and $t>0$. Moreover,
\begin{align}
\label{ineqh1}&|h_v(x)|\leq1,\\
\label{ineqh2}&|h_v(x)-h_v(y)|\leq \frac{2}{1-\ka(\de)}|x-y|^{\de},
\end{align}
for every $x, y\in\R^d$ and every $0<\de\leq 1$ such that $0<\de<\al$.
\end{lem}
\begin{proof} Inequality, (\ref{ineqh1}) is obvious, (\ref{ineqh2}) follows from the definition of function $h_v$ and
inequality $\left|\bp_k(x)-\bp_k(y)\right|\le|M||x-y|$ for $k\in\N$, where $\bp_k$ was defined in
(H1) of assumption \ref{shapefun}. In order to prove that $h_v(tx)=h_{tv}(x)$ it is enough to show
that for a fixed $s>0$ and every $x\in\R^d$, $\bp(sx)=s\bp(x)$. Indeed, for every $\eps>0$ there
exists $\eta>0$ such that $\left|t\p(t^{-1}sx)-\bp(sx)\right|<\eps$ for every $0<t<s\eta$. Hence if
$t=rs$ and $0<r<\eta$ then $\left|sr\p(r^{-1}x)-\bp(sx)\right|<\eps$. Letting $r$ tend to $0$ we
obtain
 $s\bp(x)=\bp(sx)$.
\end{proof}
\begin{lem}\label{EigenvalueT}
The unique eigenvalue of modulus $1$ for operator $T_v$ acting on $\ct$ is $1$ with the eigenspace
$\C\cdot h_v(x)$, where function $h_v$ was defined in (\ref{ET}).
\end{lem}
\begin{proof} Notice that $\lim_{n\rightarrow\8}\bp_{n}\circ\ldots\circ \bp_{1}(x)=0$ a.e., ($\bp_{k}$'s were defined in (H1) of assumption \ref{shapefun}). Take
$f\in\ct$, then by the Lebesgue dominated convergence theorem we have
\begin{align*}
T^n_vf(x)&=\E\left(e^{i\left\langle v, \sum_{k=1}^{n}\bp_{k}\circ\ldots\circ
\bp_{1}(x)\right\rangle}\left(f\left(\bp_{n}\circ\ldots\circ \bp_{1}(x)\right)-f(0)\right)\right)\\
&+\E\left(e^{i\left\langle v, \sum_{k=1}^{n}\bp_{k}\circ\ldots\circ
\bp_{1}(x)\right\rangle}f(0)\right)\ _{\overrightarrow{n\rightarrow\8}}\ f(0)h_v(x),
\end{align*}
for every $x\in\R^d$. Since $h_v(0)=1$ the above convergence shows that $1$ is a simple eigenvalue
for the action of $T_v$ on $\bl1$ with $h_v$ as the unique associated eigenfunction (up to
multiplicative constant). It also proves that $1$ is the unique peripheral eigenvalue.
\end{proof}

\begin{proof}[Proof of Proposition \ref{KL}]
Remark \ref{KL1} implies that $\ker(T_v-I)=\ker(T_v-I)^2$. By Lemmas \ref{KL2}, \ref{EigenvalueP}
and \ref{EigenvalueT}, we have thanks to \cite{H} and \cite{ITM}, that for every $n\in\N$
\begin{align*}
\forall f\in\bl1\ \ \ P^nf=P^n_{0, v}f=\nu(f)\cdot1+Q^n_{P, 0}f,\ \ \ \mbox{and}\ \ \
T^n_vf=T^n_{0, v}f=f(0)\cdot h_v+Q^n_{T, 0}f,
\end{align*}
where $Q_{P, 0}$ and $Q_{T, 0}$ are the complementary operators to the projections $\Pi_{P, 0}$
($\Pi_{P, 0}f=\nu(f)\cdot1$) and $\Pi_{T, 0}$ ($\Pi_{T, 0}f=f(0)\cdot h_v$) respectively.

Besides, in view of Lemmas \ref{KL2} and \ref{KL3}, (in particular (\ref{KL3A}) and (\ref{KL3B})
imply
\begin{align*}
\forall f\in\bl1\ \ \ \left|\left(T_{t,v}-T_v\right)f\right|_{\rho}\leq C_4t^{\ep}\|f\|_{\rel},\ \
\ \mbox{and}\ \ \ \left|\left(P_{t,v}-P\right)f\right|_{\rho}\leq C_5t^{\ep}\|f\|_{\rel},
\end{align*}
respectively for every $t\in[0, 1]$) we may use the perturbation theorem of Keller and Liverani
\cite{KL} (we refer also to \cite{Liv} for an improvement of \cite{KL}) for the operators $P_{t,
v}$ and $T_{t, v}$ to get Proposition \ref{KL}.
\end{proof}
\begin{rem}\label{KL4}
If $z\in\sigma(P_{t, v})$ or $z\in\sigma(T_{t, v})$ and $|z|>\varrho$, where $0<\varrho<1$ is
defined as in Lemma \ref{KL2}, then $z$ does not belong to the residual spectrum of the operator
$P_{t, v}$ or $T_{t, v}$  (see \cite{H} and \cite{ITM}). But thanks to the improvement of \cite{KL}
given in \cite{Liv} (see also \cite{HP}) condition on the essential spectral radius of $P_{t, v}$
(and $T_{t, v}$) is not required for $t\not=0$.
\end{rem}
\section{Rate of convergence and fractional expansions}
\subsection{Rate of convergence of projections.} As it has been already mentioned, to write down an expansion of
$k_v(t)$ sufficiently good for the limit Theorem \ref{limitthm}, we study the peripherical
eigenfunctions of $P_{t,v}$ and, when $t$ varies, the normalization is important. We have two
natural candidates $\Pi _{P,t}1 $ and
 $\Delta _t\Pi_{T, t}h_v $ one being a multiple of the other
$$
\Delta _t\Pi_{T, t}h_v(x)=c_t\Pi _{P,t}(x).$$

Notice that, if $P_{t,v}f_t=k_v(t)f_t$, then
\begin{align*}
(k_v(t)-1)\cdot\nu(f_t)=&\nu ((P_{t,v}-1)f_t)\\
=&\int \left( e^{it\langle v, \psi_{\te} (x) \rangle}-1\right)f_t(\psi_{\te} (x))\ \mu (d\te)\nu (dx)\\
=&\int\left(e^{it\langle v, x \rangle}-1\right)f_t(x)\nu (dx).
\end{align*}
Therefore, for both $\Pi _{P,t}1 $ and
 $\Delta _t\Pi_{T, t}h_v $ we have
\begin{align}
(k_v(t)-1)\cdot\nu(\Pi_{P, t}1)=\nu\left(\left(e^{it\langle v,
\cdot\rangle}-1\right)\cdot(\Pi_{P, t}1)\right),
\end{align}
and
\begin{align}
\label{Taylorformula}(k_v(t)-1)\cdot\nu(\Dt\Pi_{T,
t}f)=\nu\left(\left(e^{it\langle v,
\cdot\rangle}-1\right)\cdot(\Dt\Pi_{T, t}f)\right).
\end{align}
We may use either $\Pi _{P,t}1$ and approximate it by $1$, or $\Delta _t\Pi_{T, t}h_v $ and
approximate it by $h_v$. We choose the second possibility and we prove the following

\begin{thm}\label{diffrencePI}
Let $h_v$ be the eigenfunction for operator $T_v$ defined in (\ref{ET}). Then for any $0<\de\leq1$
such that $\ep<\de<\al$, there exist $C>0$ and $D>0$ such that
\begin{align}\label{diffPI1}
\left\|\Dt(\Pi_{T, t}-\Pi_{T, 0})h_v\right\|_{\rel}\leq Ct^{\de},
\end{align}
and
\begin{align}\label{rate1}
\nu(\Dt\Pi_{T, t}h_v-1)\leq Dt^{\de},
\end{align}
for every $0<t\leq t_0$. Moreover, for every $x\in\R^d$ and every $0<t\leq t_0$
\begin{align}\label{diffPI2}
\left|\Pi_{T, t}(h_v)(tx)-\Pi_{T, 0}(h_v)(tx)\right|\leq Ct^{\de}(1+|x|)^{\rho}.
\end{align}
\end{thm}

\begin{rem}The use of the family $\{T_{t, v}\}_{t>0}$ above is much more
efficient than that of $\{P_{t, v}\}_{t>0}$. Indeed, the difference $|P_{t, v}f(x)-Pf(x)|$ involves
the term $\left|e^{i\langle tv, \p(x)\rangle}-1\right|$ which depends on $x$,  while $|T_{t,
v}f(x)-T_vf(x)|$, leads to $\left| e^{i\langle v, \p_t(x)\rangle}- e^{i\langle v,
\bp(x)\rangle}\right|\leq|t||Q|$ independently of $x$ by assumption (L2). This is the main new idea
in comparison to \cite{BDG} and it allows to prove (\ref{diffPI2}). In (\ref{diffPI2}) $0<\de\le1 $
satisfies  $\ep<\de <\alpha $ while $\rho $ can be chosen small. This cannot be proved for $\Pi
_{P,t}$ directly. Moreover, in Subsection 6.2, we shall deduce from (\ref{diffPI2}) the following
important property
\begin{equation}\label{expan}
\int_{\R^d}\left(e^{it\langle v, x\rangle}-1\right)\left(\Pi_{T,
t}(h_v)(tx)-h_v(tx)\right)\nu (dx)=o(t^{\al})\ \ \ \mbox{as}\ \ \
t\to 0,
\end{equation}
 or $o(t^2\log t)$ when $\alpha =2$. Therefore, for $0<\al<2$
\begin{equation}\label{expans}
\frac{k_v(t)-1}{t^{\al}}\approx \frac{1}{t^{\al}} \int_{\R^d}\left(e^{it\langle v,
x\rangle}-1\right)h_v(tx)\nu(dx),\end{equation}
 when $t\to 0$ (if $\al=2$ we have $t^2|\log t|$ instead of $t^{\al}$ in the above denominators) and
the right hand side of (\ref{expans}) is further studied in the next subsection.
\end{rem}
Before we prove Theorem \ref{diffrencePI} we need two lemmas.
\begin{lem}\label{rateprojlem}
Assume that the function $f$ satisfies $|f(x)|\leq C$ for any $x\in\R^d$, and $|f(x)-f(y)|\leq
C|x-y|^{\de}$ for any $0<\de\leq1$ and $x, y\in\R^d$, where constant $C>0$ depends on $\de$. Then
for every $\de\in(\ep, \al)$ %we have following inequalities
\begin{align}
\label{rplA}&\left[(T_{t, v}-T_{v})f\right]_{\el}\leq C_1t^{\de-\ep},\\
\label{rplB}&\left|(T_{t, v}-T_{v})f\right|_{\rho}\leq C_2t^{\de},
\end{align}
where  $C_1>0$ and $C_2>0$ do not depend on $0<t\leq t_0$.
\end{lem}
\begin{proof}
In order to show (\ref{rplA}) we have to estimate the seminorm $\left[(T_{t,
v}-T_{v})f\right]_{\el}$. Notice, that
\begin{align}
\label{rpla}\left[(T_{t, v}-T_{v})f\right]_{\el}&\leq\sup_{x\not=y, |x-y|\leq t}\frac{\left|(T_{t, v}-T_{v})f(x)-(T_{t, v}-T_{v})f(y)\right|}{|x-y|^{\ep}(1+|x|)^{\la}(1+|y|)^{\la}}\\
\nonumber&+\sup_{x\not=y, |x-y|>t}\frac{\left|(T_{t, v}-T_{v})f(x)-(T_{t,
v}-T_{v})f(y)\right|}{|x-y|^{\ep}(1+|x|)^{\la}(1+|y|)^{\la}}.
\end{align}
For the first term in (\ref{rpla}) ($|x-y|\leq t$) we observe that
\begin{align}
\nonumber(T_{t, v}-T_{v})f(x)&-(T_{t, v}-T_{v})f(y)=\\
\label{rplb}&=\E\left(\left(e^{i\langle
v,\p_{t}(x)\rangle}-e^{i\langle v,\p_{t}(y)\rangle}\right)f(\p_{t}(x))\right)\\
\label{rplc}&+\E\left(e^{i\langle v,\p_{t}(y)\rangle}\left(f(\p_{t}(x))-f(\p_{t}(y))\right)\right)\\
\label{rpld}&-\E\left(\left(e^{i\left\langle v,\bp(x)\right\rangle}-e^{i\left\langle
v,\bp(y)\right\rangle}\right)f\left(\bp(x)\right)\right)\\
\label{rple}&- \E\left(e^{i\left\langle
v,\bp(y)\right\rangle}\left(f\left(\bp(x)\right)-f\left(\bp(y)\right)\right)\right).
\end{align}
%Above formula is crucial to bound the first term of (\ref{rpla}).
We will estimate (\ref{rplb}), (\ref{rplc}), (\ref{rpld}) and
(\ref{rple}) separately. By assumptions on the function $f$
observe, that for every $0<\de\leq1$ such that $\ep<\de<\al$ we
have
\begin{align}
\label{rplf}\E\left(\frac{\left|e^{i\langle v,\p_{t}(x)\rangle}-e^{i\langle
v,\p_{t}(y)\rangle}\right||f(\p_{t}(x))|}{|x-y|^{\ep}(1+|x|)^{\la}(1+|y|)^{\la}}\right)
&\leq 2C\E\left(\frac{|\p_{t}(x)-\p_{t}(y)|^{\de}}{|x-y|^{\ep}}\right)\\
\nonumber\leq 2C\E\left(|M|^{\de}\right)|x-y|^{\de-\ep}&\leq 2C\E\left(|M|^{\de}\right)t^{\de-\ep}.
\end{align}
Similarly, we obtain the estimate of the second term. Indeed,
\begin{align}
\label{rplg}\E\left(\frac{\left|e^{i\langle
v,\p_{t}(y)\rangle}(f(\p_{t}(x))-f(\p_{t}(y)))\right|}{|x-y|^{\ep}(1+|x|)^{\la}(1+|y|)^{\la}}\right)&\leq
\E\left(\frac{\left|f(\p_{t}(x))-f(\p_{t}(y))\right|}{|x-y|^{\ep}}\right)\\
\nonumber\leq 2C\E\left(|M|^{\de}\right)|x-y|^{\de-\ep}&\leq 2C\E\left(|M|^{\de}\right)t^{\de-\ep}.
\end{align}
Remaining (\ref{rpld}) and (\ref{rple}) are estimated in the similar way, since
$\left|\bp(x)-\bp(y)\right|\le|M||x-y|$ by definition of $\bp$. Now consider the second term of
(\ref{rpla}) ($|x-y|> t$) and notice, that
\begin{align}
\nonumber(T_{t, v}-T_{v})f(x)&-(T_{t, v}-T_{v})f(y)\\
\label{rplh}&=\E\left(\left(e^{i\langle v,\p_{t}(x)\rangle}-e^{i\left\langle
v,\bp(x)\right\rangle}\right)f(\p_{t}(x))\right)\\
\label{rpli}&+ \E\left(e^{i\left\langle
v,\bp(x)\right\rangle}\left(f(\p_{t}(x))-f\left(\bp(x)\right)\right)\right)\\
\label{rplj}&-\E\left(\left(e^{i\langle v,\p_{t}(y)\rangle}-e^{i\left\langle
v,\bp(y)\right\rangle}\right)f(\p_{t}(y))\right)\\
\label{rplk}&-\E\left(e^{i\left\langle
v,\bp(y)\right\rangle}\left(f\left(\p_{t}(y)\right)-f\left(\bp(y)\right)\right)\right).
\end{align}
As before, we will estimate (\ref{rplh}), (\ref{rpli}),
(\ref{rplj}) and (\ref{rplk}) separately using (L2) and (L3).
Indeed, for every $0<\de\leq1$ such that $\ep<\de<\al$ we have
\begin{align}
\label{rpll}\E\left(\frac{\left|e^{i\langle v,\p_{t}(x)\rangle}-e^{i\left\langle
v,\bp(x)\right\rangle}\right||f(\p_{t}(x))|}{|x-y|^{\ep}(1+|x|)^{\la}(1+|y|)^{\la}}\right)&\leq 2C
\E\left(\frac{\left|\p_{t}(x)-\bp(x)\right|^{\de}}{|x-y|^{\ep}}\right)\\
\nonumber\leq2C\E\left(\frac{t^{\de}|Q|^{\de}}{|x-y|^{\ep}}\right)&\leq
2C\E\left(|Q|^{\de}\right)t^{\de-\ep}.
\end{align}
Similarly, we obtain the estimate for the second term. Indeed,
\begin{align}
\label{rplm}&\E\left(\frac{\left|e^{i\left\langle
v,\bp(x)\right\rangle}\left(f(\p_{t}(x))-f\left(\bp(x)\right)\right)\right|}{|x-y|^{\ep}(1+|x|)^{\la}(1+|y|)^{\la}}\right)\leq
\E\left(\frac{\left|f(\p_{t}(x))-f\left(\bp(x)\right)\right|}{|x-y|^{\ep}}\right)\\
\nonumber&\leq2C\E\left(\frac{\left|\p_t(x)-\bp(x)\right|^{\de}}{|x-y|^{\ep}}\right)\leq
2C\E\left(\frac{t^{\de}|Q|^{\de}}{|x-y|^{\ep}}\right) \leq 2C\E\left(|Q|^{\de}\right)t^{\de-\ep}.
\end{align}
Also remaining (\ref{rplj}) and (\ref{rplk}) can be estimated
analogously. Hence, in view of (\ref{rplf}), (\ref{rplg}),
(\ref{rpll}) and (\ref{rplm}),
we obtain (\ref{rplA}). For (\ref{rplB}) %we have to estimate norm $|(T_{t, v}-T_{v})h_v|_{\rho}$. For this purpose
notice that
\begin{align}
\label{rpln}(T_{t, v}-T_{v})f(x)&=\E\left(\left(e^{i\langle v,\p_{t}(x)\rangle}-e^{i\left\langle
v,\bp(x)\right\rangle}\right)f(\p_{t}(x))\right)\\
\nonumber&+\E\left(e^{i\left\langle
v,\bp(x)\right\rangle}\left(f(\p_{t}(x))-f\left(\bp(x)\right)\right)\right).
\end{align}
We have,
\begin{align}
\label{rplo}\E\left(\frac{\left|e^{i\langle v,\p_{t}(x)\rangle}-e^{i\left\langle
v,\bp(x)\right\rangle}\right||f(\p_{t}(x))|}{(1+|x|)^{\rho}}\right)&\leq 2C
\E\left(\left|\p_{t}(x)-\bp(x)\right|^{\de}\right)\\
\nonumber \leq2C\E\left(t^{\de}|Q|^{\de}\right)&\leq 2C\E\left(|Q|^{\de}\right)t^{\de},
\end{align}
and
\begin{align}
\label{rplp}&\E\left(\frac{\left|e^{i\left\langle
v,\bp(x)\right\rangle}\left(f(\p_{t}(x))-f\left(\bp(x)\right)\right)\right|}{(1+|x|)^{\rho}}\right)
\leq \E\left(\left|f(\p_{t}(x))-f\left(\bp(x)\right)\right|\right)\\
\nonumber&\leq 2C\E\left(\left|\p_t(x)-\bp(x)\right|^{\de}\right) \leq
2C\E\left(t^{\de}|Q|^{\de}\right)\leq 2C\E\left(|Q|^{\de}\right)t^{\de}.
\end{align}
Combining (\ref{rplo}) with (\ref{rplp}) we obtain (\ref{rplB}) which completes the proof of the
Lemma.
\end{proof}
\begin{lem}\label{hololem2}
Let $h_v$ be the eigenfunction for operator $T_v$ defined in (\ref{ET}), then
\begin{align}
\label{hl2}\Dt(\Pi_{T, t}-\Pi_{T, 0})h_v=\frac{1}{2\pi i}\int_{|z-1|=\xi_1}\frac{1}{z-1}(z-P_{t,
v})^{-1}\Dt(T_{t, v}-T_v)h_vdz,
\end{align}
where $\xi_1>0$ was defined in Proposition \ref{KL}.
\end{lem}
\begin{proof}
Notice, that $(z-T_v)^{-1}h_v=\frac{1}{z-1}h_v$,
\begin{align*}
(z-T_{t, v})^{-1}-(z-T_v)^{-1}=(z-T_{t, v})^{-1}(T_{t, v}-T_{v})(z-T_v)^{-1},
\end{align*}
and by the definition (\ref{dilat})
\begin{align*}
(z-P_{t, v})^{-1}\Dt=\Dt(z-T_{t, v})^{-1},
\end{align*}
Then
\begin{align*}
\Dt(\Pi_{T, t}-\Pi_{T, 0})h_v&=\frac{1}{2\pi i}\int_{|z-1|=\xi_1}\Dt(z-T_{t, v})^{-1}(T_{t,
v}-T_{v})(z-T_v)^{-1}h_vdz\\&=\frac{1}{2\pi i}\int_{|z-1|=\xi_1}\frac{1}{z-1}\Dt(z-T_{t,
v})^{-1}(T_{t, v}-T_v)h_vdz\\
&=\frac{1}{2\pi i}\int_{|z-1|=\xi_1}\frac{1}{z-1}(z-P_{t, v})^{-1}\Dt(T_{t, v}-T_v)h_vdz,
\end{align*}
which completes the proof of (\ref{hl2}).
\end{proof}

\begin{proof}[Proof of Theorem \ref{diffrencePI}] For every $f\in\bl1$ and $|t|\leq1$ we have
\begin{align}
\label{ineqN}\|\Dt f\|_{\rel}\leq |f|_{\rho}+|t|^{\ep}[f]_{\el},
\end{align}
In view of (\ref{hl2}), Proposition \ref{KL}, inequalities
(\ref{ineqN}), (\ref{rplA}) and (\ref{rplB}) with the function
$h_v$ we have
\begin{align*}
\|\Dt(\Pi_{T, t}&-\Pi_{T,
0})h_v\|_{\rel}\leq\\
&\leq\frac{1}{2\pi\xi_1}\int_{|z-1|=\xi_1}\left\|(z-P_{t, v})^{-1}\Dt(T_{t, v}-T_v)h_v\right\|_{\rel}dz\\
&\leq D(|\Dt(T_{t, v}-T_{v})h_v|_{\rho}+[\Dt(T_{t,
v}-T_{v})h_v]_{\el})\\
&\leq D (|(T_{t, v}-T_{v})h_v|_{\rho}+t^{\ep}[(T_{t,
v}-T_{v})h_v]_{\el})\\
&\leq D (C_2t^{\de}+t^{\ep}C_1t^{\de-\ep})\leq Ct^{\de},
\end{align*}
for every $0\le t\leq t_0$ and it completes the proof of (\ref{diffPI1}). Now (\ref{diffPI2})
follows. In order to prove (\ref{rate1}) apply inequality (\ref{diffPI2}) and (\ref{ineqh2}) to
obtain
\begin{align*}
\left|(\Dt\Pi_{T, t}h_v)(x)-1\right|&\leq\left|\Pi_{T, t}(h_v)(tx)-\Pi_{T,
0}(h_v)(tx)\right|+\left|h_v(tx)-1\right|\\
&\leq t^{\de}\left(C(1+|x|)^{\rho}+\frac{2}{1-\ka(\de)}|x|^{\de}\right).
\end{align*}
Above inequality implies (\ref{rate1}) and the proof is finished.
\end{proof}
\subsection{Rate of convergence of eigenvalues and fractional expansions.}
In this section we study the fractional expansions (\ref{expans}) when $t\to 0 $. Their behavior is
strongly related to the asymptotics of the stationary measure $\nu $ at infinity. While Theorem
\ref{LT} is proved then the limit Theorem \ref{limitthm} follows as in \cite{BDG}. First we
establish (\ref{expan}).

The proof goes along the same lines as in \cite{BDG} with the
function $h_v$ playing the role of $\hat \eta _v$ there.
Therefore, the details have been omitted. We shall use radial
coordinates in $\R^d$ i.e. every point is expressed as $tv$ where
$t>0$ and $v\in\Ss^{d-1}$.

\begin{con}\label{con} Assume that $0<\ep<1$, $\la>0$, $\la+2\ep<\rho=2\la$ and $2\la+\ep<\al$ as in Proposition
\ref{KL} and additionally
\begin{itemize}
\item If $0<\al\leq1$, take any $0<\be<\frac{1}{2}$ such that $\rho+2\be<\al$.
\item If $1<\al\leq2$, take any $\la>0$ such that $\rho=2\la<1$ and $\rho+1<\al$.
\end{itemize}
\end{con}
\begin{prop}\label{dyaprop}
Let $h_v$ be the eigenfunction of operator $T_v$ defined in (\ref{ET}). If $0<\al<2$, then
\begin{align}\label{dyaA}
\lim_{t\rightarrow0}\frac{1}{t^{\al}}\int_{\R^d}\left(e^{it\langle v,
x\rangle}-1\right)\left(\Pi_{T, t}(h_v)(tx)-\Pi_{T, 0}(h_v)(tx)\right)\nu(dx)=0.
\end{align}
If $\al=2$, then
\begin{align}\label{dyaB}
\lim_{t\rightarrow0}\frac{1}{t^2|\log t|}\int_{\R^d}\left(e^{it\langle v,
x\rangle}-1\right)\left(\Pi_{T, t}(h_v)(tx)-\Pi_{T, 0}(h_v)(tx)\right)\nu(dx)=0.
\end{align}
\end{prop}
\begin{proof}
In estimations below in view of Condition \ref{con} we have to use appropriate parameters $\ep,
\la, \rho, \de$ and $\eta$ which are determined by $\alpha$.
\begin{itemize}
\item If $0<\al\leq1$, we take  $\de=\al-\be>\rho+\be>\ep$ and $\eta=2\be$.
\item If $1<\al\leq2$, we take $\de=1>\ep$ and $\eta=1$.
\end{itemize}
In view of (\ref{diffPI2}), we have
\begin{align}
\label{dyaa}\left|\frac{1}{t^{\al}}\int_{\R^d}\left(e^{it\langle v, x\rangle}-1\right)\left(\Pi_{T,
t}(h_v)(tx)-\Pi_{T, 0}(h_v)(tx)\right)\nu(dx)\right|\leq&\\
\nonumber\leq C_1t^{\eta+\de-\al}\int_{\R^d}|x|^{\eta}(1+|x|)^{\rho}\nu(dx)\leq&
C_2t^{\eta+\de-\al}.
\end{align}
for every $0< t\leq t_0\leq1$. Notice that, if
\begin{itemize}
\item $0<\al\leq1$, then $\eta+\de-\al=2\be+\al-\be-\al=\be>0$ and $\rho+\eta=\rho+2\be<\al$.
\item $1<\al\leq2$, then $\eta+\de-\al=1+1-\al=2-\al\geq0$ and $\rho+\eta=\rho+1<\al$.
\end{itemize}
This justifies inequalities in (\ref{dyaa}) and completes the proof of (\ref{dyaA}) and
(\ref{dyaB}).
\end{proof}

\begin{thm}\label{LT}
Assume that $\p_{\te}$ satisfies assumptions \ref{shapefun}, \ref{heavytail} and \ref{limitass} for
$\te\in\Te$. We define $S_n^x=\sum_{k=1}^{n}X_{k}^{x}$ for $n\in\N$. Let
$h_v(x)=\E\left(e^{i\left\langle v, \sum_{k=1}^{\8}\bp_{k}\circ\ldots\circ
\bp_{1}(x)\right\rangle}\right)$ for $x\in\R^d$, where $\bp_{k}(x)=\bp_{\te_k}(x)$, and
$\bp_{\te_k}$'s were defined in (H1) of assumption \ref{shapefun}. Measure $\nu$ is the stationary
measure for the recursion (\ref{rec}),
and $\La$ and $\si_{\La}$ are the measures defined in Theorem \ref{HTt}.\\

 \noindent\textsc{Case $0<\al<1$.} Let $\Xi_{\al}^n$ be the characteristic
function of the random variable $n^{-\frac{1}{\al}}S_n^x$. Then for every $t>0$ and $v\in\Ss^{d-1}$
\begin{align}\label{ch1}
\lim_{n\rightarrow\8}\Xi_{\al}^n(tv)=\ups_{\al}(tv)=\exp(t^{\al}C_{\al}(v)),
\end{align}
where
\begin{align}\label{k1}
C_{\al}(v)=\lim_{t\rightarrow0}\frac{k_v(t)-1}{t^{\al}}=\int_{\R^d}\left(e^{i\langle v,
x\rangle}-1\right)h_v(x)\La(dx).
\end{align}

\noindent\textsc{Case $\al=1$}. Let $\Xi_{1}^n$ be the characteristic function of the random
variable $n^{-1}S_n^x-n\xi(n^{-1})$. Then for every $t>0$ and $v\in\Ss^{d-1}$
\begin{align}\label{ch2}
\lim_{n\rightarrow\8}\Xi_{1}^n(tv)=\ups_{1}(tv)=\exp(tC_{1}(v)+it\langle v, \tau(t)\rangle),
\end{align}
where
\begin{align}\label{k2}
C_{1}(v)=\lim_{t\rightarrow0}\frac{k_v(t)-1-i\langle v,
\xi(t)\rangle}{t}=\int_{\R^d}\left(\left(e^{i\langle v, x\rangle}-1\right)h_v(x)-\frac{i\langle v,
x\rangle}{1+|x|^2}\right)\La(dx),
\end{align}
$\xi(t)=\int_{\R^d}\frac{tx}{1+|tx|^2}\nu(dx)$ and
$\tau(t)=\int_{\R^d}\left(\frac{x}{1+|tx|^2}-\frac{x}{1+|x|^2}\right)\La(dx)$.

\noindent\textsc{Case $1<\al<2$.} Assume that $m=\int_{\R^d}x\nu(dx)$. Let $\Xi_{\al}^n$ be the
characteristic function of the random variable $n^{-\frac{1}{\al}}\left(S_n^x-nm\right)$. Then for
every $t>0$ and $v\in\Ss^{d-1}$
\begin{align}\label{ch3}
\lim_{n\rightarrow\8}\Xi_{\al}^n(tv)=\ups_{\al}(tv)=\exp(t^{\al}C_{\al}(v)),
\end{align}
where
\begin{align}\label{k3}
C_{\al}(v)=\lim_{t\rightarrow0}\frac{k_v(t)-1-i\langle v,
tm\rangle}{t^{\al}}=\int_{\R^d}\left(\left(e^{i\langle v, x\rangle}-1\right)h_v(x)-i\langle v,
x\rangle\right)\La(dx).
\end{align}

\noindent\textsc{Case $\al=2$.} Assume that $m=\int_{\R^d}x\nu(dx)$. Let $\Xi_{2}^n$ be the
characteristic function of the random variable $\left(n\log
n\right)^{-\frac{1}{2}}\left(S_n^x-nm\right)$. Then for every $t>0$ and $v\in\Ss^{d-1}$
\begin{align}\label{ch4}
\lim_{n\rightarrow\8}\Xi_{2}^n(tv)=\ups_{2}(tv)=\exp(t^2C_{2}(v)),
\end{align}
where
\begin{align}\label{k4}
C_{2}(v)=\lim_{t\rightarrow0}\frac{k_v(t)-1-i\langle v, tm\rangle}{2t^{2}|\log
t|}=-\frac{1}{4}\int_{\Ss^{d-1}}\left(\langle v, w\rangle^2+2\langle v, w\rangle\langle v,
\E(\vp(w))\rangle\right)\si_{\La}(dw),
\end{align}
and $\vp(x)=\sum_{k=1}^{\8}\bp_{k}\circ\ldots\circ \bp_{1}(x)=\sum_{k=1}^{\8}M_{k}\cdot\ldots\cdot
M_{1}x$, where $M_{k}$'s were defined in (H1) of assumption \ref{shapefun}.

Moreover, $C_{\al}(tv)=t^{\al}C_{\al}(v)$ for every $t>0$, $v\in\Ss^{d-1}$ and $\al\in(0, 1)\cup(1,
2]$. If
%$\vp(x)=\sum_{k=1}^{\8}M_{k}\cdot\ldots\cdot M_{1}x$ for every $x\in\supp\La$ and
$\supp\si_{\La}$ spans $\R^d$ as a linear space, then $\Re C_{\al}(v)<0$ for every $v\in\Ss^{d-1}$
and $\al\in(0, 2]$.
\end{thm}
\begin{proof}
In order to obtain fractional expansions (\ref{k1}), (\ref{k2}), (\ref{k3}) and (\ref{k4}), we have
to proceed as in Theorem 5.1 from \cite{BDG} (see also \cite{GL1}) using formula
(\ref{Taylorformula}), Proposition \ref{dyaprop},
 inequality (\ref{rate1}) and convergence (\ref{HTthm11}) which holds
for every function from family $\F_{\La}$. Proof of (\ref{ch1}), (\ref{ch2}), (\ref{ch3}) and
(\ref{ch4}) base on section 6 from \cite{BDG} (see also \cite{GL1}). Using formula
$h_v(tx)=h_{tv}(x)$ from Lemma \ref{ETL} we obtain that $C_{\al}(tv)=t^{\al}C_{\al}(v)$ is valid
for every $t>0$, $v\in\Ss^{d-1}$ and $\al\in(0, 1)\cup(1, 2]$. Method developed in section 5.6 in
\cite{BDG} allows us to show that $\Re C_{\al}(v)<0$ for every $v\in\Ss^{d-1}$ and $\al\in(0, 2]$.
\end{proof}

\section*{Acknowledgements}
The results of this paper are part of the author's PhD thesis, written under the supervision of
Ewa Damek at the University of Wroclaw. %from whom I learned all I know on stochastic recursions.
I wish to thank her for many stimulating conversations and several helpful suggestions during the
preparation of this paper. I would like also to thank Dariusz Buraczewski for beneficial
discussions and comments.

The author is grateful to the referee for a very careful reading
of the manuscript and useful remarks that lead to the improvement
of the presentation.

\end{document}